\documentclass[12pt]{amsart}
\usepackage{amsfonts,amsmath, amssymb, hyperref}
\usepackage[shortlabels]{enumitem}
\usepackage[dvips]{graphicx} 
\usepackage{epsfig,psfrag,t4phonet, expdlist }
\usepackage[english]{babel}
\usepackage[]{inputenc}
\usepackage{textcomp} 
\usepackage[all]{xy}
\newtheorem{teo}{Theorem}[section]
\newtheorem{teoi}{Theorem}[section]
\newtheorem{fact}[teo]{Fact}

\newtheorem{problem}[teo]{Problem}
\newtheorem{Question}[teo]{Question}
\newtheorem{prop}[teo]{Proposition}
 
\newtheorem{Lemma}[teo]{Lemma}
\newtheorem{Cor}[teo]{Corollary}

\newtheorem{claim}{Claim}[teo]
\newtheorem{remark}[teo]{Remark}
\newtheorem{Notation}[teo]{Notation}
\newtheorem{defin}[teo]{Definition}
\newcommand{\mc}[1]{\mathcal{#1}}
\newcommand{\norm}[1]{\left\Vert#1\right\Vert}
\newcommand{\abs}[1]{\left\vert#1\right\vert}
\newcommand{\set}[1]{\left\{\,#1\,\right\}}
\newcommand{\defined}[1]{\emph{#1}}

 \DeclareMathOperator{\spann}{span}

\newcommand{\er}{\mathbb{R}}
\newcommand{\en}{\mathbb{N}}
 
\newcommand{\zet}{\mathbb{Z}} 
\newcommand{\qu}{\mathbb{Q}}
\newcommand{\ce}{\mathbb{C}}

\date{\today}
\author{Christopher J. Eagle \and Alessandro Vignati}
\thanks{The first author was partially supported by an Ontario Graduate Scholarship award.}
\title[Saturation and elementary equivalence of C*-algebras]{Saturation and elementary equivalence of C*-algebras}
\address[Christopher J. Eagle]{University of Toronto, Department of Mathematics, 40 St. George St., Toronto, Ontario, Canada M5S 2E4}
\email{cjeagle@math.toronto.edu}%
\address[Alessandro Vignati]{York University, Department of Mathematics and Statistics, 4700 Keele St., Toronto, Ontario, Canada M3J 1P3}
\email{ale.vignati@gmail.com}

\begin{document}
\maketitle
\begin{abstract}
We study the saturation properties of several classes of C*-algebras.  Saturation has been shown by Farah and Hart to unify the proofs of several properties of coronas of $\sigma$-unital C*-algebras; we extend their results by showing that some coronas of non-$\sigma$-unital C*-algebras are countably degree-$1$ saturated.  We then relate saturation of the abelian C*-algebra $C(X)$, where $X$ is $0$-dimensional, to topological properties of $X$, particularly the saturation of $CL(X)$.   We also characterize elementary equivalence of the algebras $C(X)$ in terms of $CL(X)$ when $X$ is $0$-dimensional, and show that elementary equivalence of the generalized Calkin algebras of densities $\aleph_\alpha$ and $\aleph_\beta$ implies elementary equivalence of the ordinals $\alpha$ and $\beta$.
\end{abstract}

\section{Introduction}
In this paper we examine extent to which several classes of operator algebras are saturated in the sense of model theory.  In fact, few operator algebras are saturated in the full model-theoretic sense, but in this setting there are useful weakenings of saturation that are enjoyed by a variety of algebras.  The main results of this paper show that certain classes of C*-algebras do have some degree of saturation, and as a consequence, have a variety of properties previously considered in the operator algebra literature. For all the definitions involving continuous model theory for metric structures (or in particular of C*-algebras), we refer to \cite{BenYaacov2008a} or \cite{FarahHartSherman}. Different degrees of saturation and relevant concepts will be defined in Section \ref{section:Background}.

Among the weakest possible kinds of saturation an operator algebra may have, which nevertheless has interesting consequences, is being \textit{countably degree-$1$ saturated}.  This property was introduced by Farah and Hart in \cite{farah2011countable}, where it was shown to imply a number of important consequences (see Theorem \ref{thm:ConsequencesOfSat} below).  It was also shown in \cite{farah2011countable} that countable degree-$1$ saturation is enjoyed by a number of familiar algebras, such as coronas of $\sigma$-unital C*-algebras and all non-trivial ultraproducts and ultrapowers of C*-algebras.  Further examples were found by Voiculescu \cite{Voiculescu}.  Countable degree-$1$ saturation can thus serve to unify proofs about these algebras.  We extend the results of Farah and Hart by showing that a class of algebras which is broader than the class of $\sigma$-unital ones have countably degree-$1$ saturated coronas.  The following theorem is Theorem \ref{thm:FactorCtbleSat} below; for the definitions of $\sigma$-unital C*-algebras and essential ideals, see Definition \ref{def:EssIdeal}.

\begin{teoi}\label{TheoremA}
Let $M$ be a unital C*-algebra, and let $A \subseteq M$ be an essential ideal.  Suppose that there is an increasing sequence of positive elements in $A$ whose supremum is $1_M$, and suppose that any increasing uniformly bounded sequence converges in $M$.  Then $M/A$ is countably degree-$1$ saturated.
\end{teoi}

Theorem \ref{TheoremA} is proved as Theorem \ref{thm:FactorCtbleSat} below.  One interesting class of examples of a non-$\sigma$-unital algebra to which our result applies is the following.  Let $N$ be a II$_1$ factor, $H$ a separable Hilbert space and $\mc{K}$ be the unique two-sided closed ideal of the von Neumann tensor product $N\,\overline{\otimes}\,\mc{B}(H)$ (see \cite{Breu168} and \cite{Breu269}).  Then $(N\,\overline{\otimes} \,\mc{B}(H)) / \mc{K}$ is countably degree-$1$ saturated.  These results are the contents of Section \ref{section:Calkin}.

In Section \ref{section:Calkin2} we consider generalized Calkin algebras of uncountable weight, as well as $\mc{B}(H)$ where $H$ has uncountable density.  Considering their complete theories as metric structures, we obtain the following (Theorem \ref{teo:generalizedcalkin} below):

\begin{teoi}\label{thmB}
Let $\alpha \neq \beta$ be ordinals, $H_\alpha$ the Hilbert space of density $\aleph_\alpha$. Let $\mathcal B_\alpha=\mathcal B(H_\alpha)$ and $\mathcal C_\alpha=\mathcal B_\alpha/\mathcal K$ the Calkin algebra of density $\aleph_\alpha$. Then the projections of the algebras $\mathcal C_\alpha$ and $\mathcal C_\beta$ as posets with respect to the Murray-von Neumann order are elementary equivalent if and only if $\alpha=\beta\mod\omega^\omega$,  where $\omega^\omega$ is computed by ordinal exponentiation, as they are the infinite projections of $\mathcal B_\alpha$ and $\mathcal B_\beta$. Consequently, if $\alpha\not\equiv\beta$ then $\mathcal B_\alpha\not\equiv \mathcal B_\beta$ and $\mathcal C_\alpha\not\equiv\mathcal C_\beta$.
\end{teoi}

Elementary equivalence of C*-algebras $A$ and $B$ can be understood, via the Keisler-Shelah theorem for metric structures, as saying that $A$ and $B$ have isomorphic ultrapowers (see Theorem \ref{thm:KeislerShelah} below for a more precise description).

For our second group of results we consider (unital) abelian C*-algebras, which are all of the form $C(X)$ for some compact Hausdorff space $X$.  We focus in particular on the real rank zero case, which corresponds to $X$ being $0$-dimensional.  In Section \ref{section:Topology} we first establish a correspondence between the Boolean algebra of the clopen set of $X$ and the theory of $C(X)$ (see Theorem \ref{boolean0dim}).
\begin{teoi}\label{thmC}
Let $X$ and $Y$ be compact $0$-dimensional Hausdorff spaces.  Then $C(X)$ and $C(Y)$ are elementarily equivalent if and only if the Boolean algebras $CL(X)$ and $CL(Y)$ are elementarily equivalent.
\end{teoi}
We obtain several corollaries of the above theorem.  For example, we show that many familiar spaces have function spaces which are elementarily equivalent, and hence have isomorphic ultrapowers.

Finally, we study saturation properties in the abelian setting.  We find that if $C(X)$ is countably degree-$1$ saturated then $X$ is a sub-Stonean space without the countable chain condition and which is not Rickart.  In the $0$-dimensional setting we describe the relation between the saturation of $C(X)$ and the saturation of $CL(X)$.  While some implications hold in general, a complete characterization occurs in the case where $X$ has no isolated points.  The following is a special case of Theorems \ref{theorem1} and \ref{theorem2} in Section \ref{section:Topology}.
\begin{teoi}\label{thmD}
Let $X$ be a compact $0$-dimensional Hausdorff space without isolated points.  Then the following are equivalent:
\begin{itemize}
\item{
$C(X)$ is countably degree-$1$ saturated,
}
\item{
$C(X)$ is countably saturated,
}
\item{
$CL(X)$ is countably saturated.
}
\end{itemize}
\end{teoi}
Before beginning the technical portion of the paper, we wish to give further illustrations of the importance of the saturation properties we will be considering, particularly the full model-theoretic notion of saturation (see Definition \ref{cd1s} below).  For countable degree-$1$ saturation we refer to Theorem \ref{thm:ConsequencesOfSat} for a list of consequences.  The following fact follows directly from the fact that axiomatizable properties are preserved to ultrapowers, which are countably saturated (see \cite[Theorem 5.4 and Proposition 7.6]{BenYaacov2008a}).
\begin{fact}
Let $P$ be a property that may or may not be satisfied by a C*-algebra.  Suppose that countable saturation implies the negation of $P$. Then $P$ is not axiomatizable (in the sense of \cite{BenYaacov2008a}).
\end{fact}
Other interesting consequences follow when the Continuum Hypothesis is also assumed. In this case, all ultrapowers of a separable algebra by a non-principal ultrafilter on $\en$ are isomorphic.  In fact, all that is needed is that the ultrapowers are countably saturated and elementarily equivalent:
\begin{fact}[{~\cite[Proposition 4.13]{FarahHartSherman}}]
Assume the Continuum Hypothesis.  Let $A$ and $B$ be two elementary equivalent countably saturated C*-algebra of density $\aleph_1$. Then $A\cong B$.
\end{fact}
Applying Parovicenko's Theorem (see \cite{Parovicenko}), the above fact immediately yields that under the Continuum Hypothesis if $X$ and $Y$ are locally compact Polish $0$-dimensional spaces then $C(\beta X \setminus X) \cong C(\beta Y \setminus Y)$.

Saturation also has consequences for the structure of automorphism groups:
\begin{fact}[{~\cite[Theorem 3.1]{farah2014rigidity}}] \label{teo:CHauto}
Assume the Continuum Hypothesis.  Let $A$ be a countably saturated C*-algebra of density $\aleph_1$. Then $A$ has $2^{\aleph_1}$-many automorphisms.  In particular, $A$ has outer automorphisms.
\end{fact}
Is it known that for Fact \ref{teo:CHauto} the assumption of countable saturation can be weakened in some particular cases (see \cite[Theorem 1.4]{coskey2012automorphisms} and \cite[Theorem 2.13]{farah2011countable}), and the property of having many automorphisms under the Continuum Hypothesis is shared by many algebras that are not even quantifier free saturated (for example the Calkin algebra). In particular it is plausible that the assumption of countable saturation in Fact \ref{teo:CHauto} can be replaced with a lower degree of saturation.

In light of this, and since the consistency of the existence of non-trivial homeomorphisms (see \cite{farah2014rigidity} for the definition) of spaces of the form $\beta\er^n\setminus\er^n$ is still open (for $n\geq 2$), it makes sense to ask about the saturation of $C(\beta \er^n\setminus \er^n)$ - see Question \ref{betaern} below.   In the opposite direction, the Proper Forcing Axiom has been used to show the consistency of all automorphisms of certain algebras being inner.  For more on this topic, see \cite{farah2014rigidity}, \cite{farah2012homeomorphisms} and \cite{farah2011all}.

\section*{Acknowledgements}
We want to first thank Ilijas Farah for countless helpful remarks and suggestions, which led to significant improvements of the results of this paper.  We also received valuable comments from the participants in the Workshop on Model Theory and Operator Algebras held in M\"unster in July 2014.  The second author thanks the organizers of that conference for their invitation to participate and their financial support.  We further thank Ilijas Farah and Bradd Hart for allowing us to include their unpublished proof of Theorem \ref{thm:quantifierElimination}.

\section{Countable degree-$1$ saturation}\label{section:Background}
In this section we describe the notions we will be considering throughout the remainder of the paper.  We also take this opportunity to fix notation that will be used in subsequent sections.  The main topics of this paper are several weakenings of the model-theoretic notion of saturation.  We begin by reviewing the definition and basic properties.  Since finite-dimensional C*-algebras have full model-theoretic saturation, and hence have all of the weakenings in which we are interested, we assume throughout the paper that all C*-algebras under discussion are infinite dimensional unless otherwise specified.

\begin{Notation}
For a compact set $K\subseteq\er$ and $\epsilon>0$, we denote the $\epsilon$-thickening of $K$ by $(K)_\epsilon=\set{x\in\er \colon d(x,K)<\epsilon}$.
\end{Notation}

We will be considering C*-algebras as structures for the continuous logic formalism of \cite{BenYaacov2008a} (or, for the more specific case of operator algebras, \cite{FarahHartSherman}).  Nevertheless, for many of our results it is not necessary to be familiar with that logic.   Informally, a \emph{formula} is an expression obtained from a finite set of norms of $*$-polynomials with complex coefficients by applying continuous functions and taking suprema and infima over some of the variables.  A formula is \emph{quantifier-free} if it does not involve suprema or infima.  A formula is a \emph{sentence} if every variable appears in the scope of a supremum or infimum.  We refer the reader to \cite{FarahHartSherman} for the precise definitions.

Given a $C^*$-algebra $A$ we will denote as $A_{\leq1}$, $A_1$ and $A^+$ the closed unit ball of $A$, its boundary, and the cone of positive elements respectively.

\begin{defin}\label{cd1s}
Let $A$ be a C*-algebra, and let $\Phi$ be a collection of formulas in the language of C*-algebras.  We say that $A$ is \defined{countably $\Phi$-saturated} if for every sequence $(\phi_n)_{n \in \en}$ of formulas from $\Phi$ with parameters from $A_{\leq 1}$, and sequence $(K_n)_{n \in \en}$ of compact sets, the following are equivalent:
\begin{enumerate}
\item[(1)] There is a sequence $(b_k)_{k \in \mathbb{N}}$ of elements of $A_{\leq 1}$ such that $\phi_n^A(\overline{b}) \in K_n$ for all $n \in \en$,
\item[(2)] For every $\epsilon > 0$ and every finite $\Delta \subset \en$ there is $(b_k)_{k \in \mathbb{N}} \subseteq A_{\leq 1}$, depending on $\epsilon$ and $\Delta$, such that $\phi_n^A(\overline{b}) \in (K_n)_{\epsilon}$ for all $n \in \Delta$.
\end{enumerate}
The three most important special cases for us will be the following:
\begin{itemize}
\item{
If $\Phi$ contains all $1$-degree $^*$-polynomials, we say that $A$ is \defined{countably $1$-degree saturated}.
}
\item{
If $\Phi$ contains all quantifier free formulas, we say that $A$ is \defined{quantifier-free saturated}.}
\item{
If $\Phi$ is the set of all formulas we say that the algebra $A$ is \defined{countably saturated}.
}
\end{itemize}
\end{defin}

Clearly condition (1) in the definition always implies condition (2), but the converse does not always hold.  We recall the (standard) terminology for the various parts of the above definition.  A set of conditions satisfying (2) in the definition is called a \emph{type}; we say that the conditions are \emph{approximately finitely satisfiable} or \emph{consistent}.  When condition (1) holds, we say that the type is \emph{realized} (or \emph{satisfied}) by $(b_k)_{k \in \mathbb{N}}$.

An equivalent definition of quantifier-free saturation is obtained by allowing only $^*$-polynomials of degree at most $2$ \cite[Lemma 1.2]{farah2011countable}.  By (model-theoretic) compactness the concepts defined by Definition \ref{cd1s} are unchanged if each compact set $K_n$ is assumed to be a singleton.  

In the setting of logic for C*-algebras, the analogue of a finite discrete structure is a C*-algebra with compact unit ball, that is, a finite-dimensional algebra.  The following fact is then the C*-algebra analogue of a well-known result from discrete logic.

\begin{fact}[{~\cite[Proposition 7.8]{BenYaacov2008a}}]\label{fact:FDSat}
Every ultraproduct of C*-algebras over a countably incomplete ultrafilter is countably saturated.  In particular, every finite-dimensional C*-algebra is countably saturated.
\end{fact}

The second part of the fact follows from the first because any ultrapower of a finite-dimensional C*-algebra is isomorphic to the original algebra (see \cite[p. 24]{BenYaacov2008a}).

A condition very similar to the countable saturation of ultraproducts was considered by Kirchberg and R\o rdam under the name ``$\epsilon$-test" in \cite[Lemma 3.1]{KirchbergRordam}.  Before returning to the analysis of the different degrees of saturation, we give definitions for two well-known concepts that we are going to use strongly, but that may not be familiar to a C*-algebraist.

\begin{defin}
The \emph{theory} of a C*-algebra $A$ is the set of all sentences in the language of C*-algebras which have value $0$ when evaluated in $A$.  We say that C*-algebras $A$ and $B$ are \defined{elementary equivalent}, written $A \equiv B$, if their theories are equal.
\end{defin}
Elementary equivalence can be defined without reference to continuous logic by way of the following result, which is known as the Keisler-Shelah theorem for metric structures.  The version we are using is stated in \cite[Theorem 5.7]{BenYaacov2008a}, and was originally proved in an equivalent setting in \cite[Theorem 10.7]{Henson2003}.
\begin{teo}\label{thm:KeislerShelah}
Let $A$ and $B$ be C*-algebras.  Then $A \equiv B$ if and only if there is an ultrafilter $\mc{U}$ (over a possibly uncountable set) such that the ultrapowers $A^{\mc{U}}$ and $B^{\mc{U}}$ are isomorphic.
\end{teo}

\begin{defin}
Let $A$ be a C*-algebra.  We say that the theory of $A$ \defined{has quantifier elimination} if for any formula $\phi(\overline x)$ and any $\epsilon>0$ there is quantifier-free formula $\psi(\overline x)$ such that for every C*-algebra $B$ satisfying $A \equiv B$, and any $\overline b\subseteq B$ (of the appropriate length) we have that in $B$,
\[\abs{\phi(\overline b)-\psi(\overline b)}\leq\epsilon.\]
\end{defin}

Countable degree-$1$ saturation is the weakest form of saturation that we will consider in this paper.  Even this modest degree of saturation for a C*-algebra has interesting consequences. In particular it implies several properties (see the detailed definition before) that were shown to hold in coronas of $\sigma$-unital algebras in \cite{pedersencorona}
\begin{defin}[\cite{pedersencorona}] Let $A$ be a C*-algebra. Then $A$ is said to be
\begin{itemize}
\item $SAW^*$ if any two $\sigma$-unital subalgebras $C,B$ are orthogonal (i.e., $bc=0$ for all $b\in B$ and $c\in C$)  if and only if are separated (i.e., there is $x\in A$ such that $xbx=b$ for all $b\in B$ and $xc=0$ for all $x\in C$).
\item $AA-CRISP$ if for any sequences of positive elements $(a_n),(b_n)$ such that for all $n$ we have $a_n\leq a_{n+1}\leq\ldots\leq b_{n+1}\leq b_n$ and any separable $D\subseteq A$ such that for all $d\in D$ we have $\lim_n\norm{[d,a_n]}=0$, there is $c\in A^+$ such that $a_n\leq c\leq  b_n$ for any $n$ and for all $d\in D$ we have $[c,d]=0$. 
\item $\sigma$-sub-Stonean if whenever $C\subseteq A$ is separable and $a,b\in A^+$ are such that $aCb=\{0\}$ then there are contractions $f,g\in C'\cap A$ such that $fg=0$, $fa=a$ and $gb=b$, $C'\cap A$ denoting the relative commutant of $C$ inside $A$.
\end{itemize}
\end{defin}

\begin{teo}[{\cite[Theorem 1]{farah2011countable}}]\label{thm:ConsequencesOfSat}
Let $A$ be a countably degree-$1$ saturated C*-algebra.  Then:
\begin{itemize}
\item{$A$ is SAW${}^*$, }
\item{$A$ is AA-CRISP,}
\item{$A$ satisfies the conclusion of Kasparov's technical theorem (see \cite[Prop. 2.8]{farah2011countable}),}
\item{$A$ is $\sigma$-sub-Stonean,}
\item{every derivation of a separable subalgebra of $A$ is of the form $\delta_b(x)= bx-xb$ for some $b \in A$}
\item {$A$ is not the tensor product of two infinite dimensional C*-algebras (this is a consequence of being $SAW^*$, see \cite{saeed}).}
\end{itemize}
\end{teo}

It is useful to know that when a degree-$1$ type can be approximately finitely satisfied by elements of a certain kind then the type can be realized by elements of the same kind.

\begin{Lemma}[{\cite[Lemma 2.1]{farah2011countable}}]\label{lem:SatisfiableByPositive}
Let $A$ be a countably degree-$1$ saturated C*-algebra.  If a type can be finitely approximately satisfied by self-adjoint elements then it can be realized by self-adjoint elements, and similarly with ``self-adjoint" replaced by ``positive".
\end{Lemma}

We will also make use of the converse of the preceding lemma, which says that to check countable degree-$1$ saturation it is sufficient to check that types which are approximately finitely satisfiable by positive elements are realized by positive elements.

\begin{Lemma}\label{positivepartial}
Suppose that $A$ is a C*-algebra that is not countably degree-$1$ saturated. Then there is a countable degree-$1$ type which is approximately finitely satisfiable by positive elements of $A$ but is not realized by any positive element of $A$.
\end{Lemma}
\begin{proof}
Let $(P_n(\overline{x}))_{n \in \mathbb{N}}$ be degree-$1$ polynomials, and $(K_n)_{n \in \mathbb{N}}$ compact sets, such that the type $\set{\norm{P_n(\overline{x})} \in K_n : n \in \mathbb{N}}$ is approximately finitely satisfiable but not satisfiable in $A$.  For each variable $x_k$, we introduce new variables $v_k, w_k, y_k$, and $z_k$.  For each $n$, let $Q_n(\overline{v}, \overline{w}, \overline{y}, \overline{z})$ be the polynomial obtained by replacing each $x_k$ in $P_n$ by $v_k + iw_k - y_k - iz_k$.  Since every $x \in A$ can be written as $x = v+iw-y-iz$ where $v, w, y, z \in A^+$, it follows that $\set{\norm{Q_n(\overline{v}, \overline{w}, \overline{y}, \overline{z})} \in K_n : n \in \mathbb{N}}$ is approximately finitely satisfiable (respectively, satisfiable) by positive elements in $A$ if and only if $\set{\norm{P_n(\overline{x})} \in K_n : n \in \mathbb{N}}$ is approximately finitely satisfiable (respectively, satisfiable).
\end{proof}

The first example of an algebra which fails to be countably degree-$1$ saturated is $\mathcal B(H)$, where $H$ is an infinite dimensional separable Hilbert space.  In fact, no infinite dimensional separable C*-algebra can be countably degree-$1$ saturated; this was observed in \cite[Proposition 2.3]{farah2011countable}.  We include here a proof of the slightly stronger result, enlarging the class of algebras that are not countably degree-$1$ saturated.

\begin{defin}\label{def:sccc}
A C*-algebra $A$ has \defined{few orthogonal positive elements} if every family of pairwise orthogonal positive elements of $A$ of norm $1$ is countable.
\end{defin}

\begin{remark}This condition was introduced recently in \cite{Masumoto} under the name \emph{strong countable chain condition}, where it was expressed in terms of the cardinality of a family of pairwise orthogonal hereditary ${}^*$-subalgebras. On the other hand, in general topology this name was already introduced by Hausdorff in a different context, so we have given this property a new name to avoid overlaps.  In the non-abelian case, is is not known whether or not this condition coincides with the notion of countable chain condition for any partial order.
\end{remark}

In Section \ref{section:Topology} of this paper we will further examine the property of having few orthogonal positive elements in the context of abelian C*-algebras.  For now, we have the following lemma:

\begin{Lemma}\label{lem:scccbad}
If an infinite dimensional C*-algebra $A$ has few orthogonal positive elements, then $A$ is not countably degree-$1$ saturated.
\end{Lemma}
\begin{proof}
Suppose to the contrary that $A$ has few orthogonal positive elements and is countably degree-$1$ saturated.  Using Zorn's lemma, find a set $Z \subseteq A_1^+$ which is maximal (under inclusion) with respect to the property that if $x, y \in Z$ and $x \neq y$, then $xy = 0$.  By hypothesis, the set $Z$ is countable; list it as $Z = \{a_n\}_{n \in \mathbb{N}}$.

For each $n \in \mathbb{N}$, define $P_n(x) = a_n x$, and let $K_n = \{0\}$.  Let $P_{-1}(x) = x$, and $K_{-1} = \{1\}$.  The type $\set{\norm{P_n(x)} \in K_n : n \geq -1}$ is finitely satisfiable.  Indeed, by definition of $Z$, for any $m \in \mathbb{N}$ and any $0 \leq n \leq m$ we have $\norm{P_n(a_{m+1})} = \norm{a_na_{m+1}} = 0$, and $\norm{a_{m+1}} = 1$.  By countable degree-$1$ saturation and Lemma \ref{lem:SatisfiableByPositive} there is a positive element $b \in A_1^+$ such that $\norm{P_n(b)} = 0$ for all $n \in \mathbb{N}$.  This contradicts the maximality of $Z$.
\end{proof}

Subalgebras of $\mc{B}(H)$ clearly have few positive orthogonal elements, whenever $H$ is separable.  As a result, we obtain the following.

\begin{Cor}\label{Hnotseparable}
No infinite dimensional subalgebra of $\mc{B}(H)$, with $H$ separable, can be countably degree-$1$ saturated.
\end{Cor}

Corollary \ref{Hnotseparable} shows that many familiar C*-algebras fail to be countably degree-$1$ saturated.  In particular, it implies that no infinite dimensional separable C*-algebra is countable degree-$1$ saturated.  Corollary \ref{Hnotseparable} also shows that the class of countably degree-$1$ saturated algebras is not closed under taking inductive limits (consider, for example, the CAR $\bigotimes_{i=1}^{\infty}M_2(\mathbb C)$, or any AF algebra) or subalgebras.  On the other hand, several examples of countably degree-$1$ saturated algebras are known.  It was shown in \cite{farah2011countable} that every corona of a $\sigma$-unital algebra is countably degree-$1$ saturated.  Recently Voiculescu in \cite{Voiculescu} found examples of  algebras which are not C*-algebras, but which have the unexpected property that their coronas are countably degree-$1$ saturated C*-algebras.  The results of the following section expand the list of examples of countably degree-$1$ saturated C*-algebras.

\section{Coronas of non-$\sigma$-unital algebras}\label{section:Calkin}
Our goal in this section is to give examples of coronas of non-$\sigma$-unital algebras which are countably degree-$1$ saturated.    For convenience, we recall some definitions which we will need:

\begin{defin}\label{def:EssIdeal}
A C*-algebra $A$ is \defined{$\sigma$-unital} (see \cite[II.4.2.4]{BlackadarOpAlg}) if it has a countable approximate identity, that is, a sequence $(e_n)_{n \in \en}$ such that for all $x \in A$,
\[\lim_{n\to\infty}\norm{e_nx - x} = \lim_{n \to \infty}\norm{xe_n-x} = 0.\]
A closed ideal $I \subseteq A$ is \defined{essential} (see \cite[II.5.4.7]{BlackadarOpAlg}) if it has trivial annihilator, that is, if $\set{x \in A : Ix = \{0\}} = \{0\}$.
\end{defin}

\subsection*{Motivating example}
\begin{Notation}
Let $\mathcal R$ be the hyperfinite II$_1$ factor.  Let $M=\mathcal R\,\overline\otimes\, \mathcal B(H)$ be the unique hyperfinite II$_\infty$ factor associated to $\mathcal R$, and let $\tau$ be its unique trace.  We denote by $\mathcal K_M$ the unique norm closed two-sided ideal generated by the positive elements of finite trace in $M$. 
\end{Notation}
 Note that $M$ is the multiplier algebra of $\mathcal K_M$, so the quotient $M / \mc{K}_M$ is the corona of $M$. 

Any ideal in a von Neumann algebra is generated, as a linear space, by its projections, hence $\mathcal K_M$ is the closure of the linear span in $M$ of the set of projections of finite trace.  In particular, $\mathcal R\otimes \mathcal K(H)\subsetneq \mathcal K_M$.  To see that the inclusion is proper, fix an orthonormal basis $(e_n)_{n \in \mathbb{N}}$ for $H$, and choose $(p_n)_{n \in \mathbb{N}}$ from $\mathcal R$ such that $\tau(p_n)=2^{-n}$ for all $n \in \mathbb{N}$.  For each $n$, let $q_{n}\in \mathcal B(H)$ be the projection onto $e_n$, and let $q = \sum_n p_n\otimes q_n$.  Then $q \in M$ is a projection of finite trace, but $q\notin \mathcal R\otimes \mathcal K(H)$.

We recall few well known properties of this object.
\begin{prop}[\cite{Phil90}]
\begin{enumerate}
\item $\mathcal R\cong M_p(\mathcal R)$ for every prime number $p$. Consequently $M_n(\mathcal R)\cong M_m(\mathcal R)$ for every $m,n\in\en$.
\item $K_0(\mathcal K_M)=\er=K_1(M/\mathcal K_M)$.
\item $\mathcal K_M$ is not $\sigma$-unital.
\item $\mathcal K_M\otimes \mathcal K(H)$ is not isomorphic to $\mathcal R\otimes\mathcal K(H)$.
\end{enumerate}
\end{prop}
\begin{proof}
\begin{enumerate}
\item This is because $M_p(\mathcal R)$ is hyperfinite and $\mathcal R$ is the unique hyperfinite II$_1$-factor.
\item Note that $K_0(M)=0=K_1(M)$ and apply the exactness of the six term $K$-sequence.
\item Suppose to the contrary that $(x_n)_{n \in \mathbb{N}}$ is a countable approximate identity in $\mc{K}_M$ formed by positive elements such that $0\leq x_n\leq 1$ for all $n$. Using spectral theory, we can find projections $p_n\in \mc{K}_M$ such that $\norm{p_nx_n-x_n}\leq 1/n$ for each $n$.  Then $(p_n)_{n \in \mathbb{N}}$ is again a countable approximate identity for $\mc{K}_M$.  For each $n \in \en$ define $q_n=\sup_{k\leq n} p_k\in \mc{K}_M$, and by passing to a subsequence we can suppose that $(q_n)_{n \in \en}$ is strictly increasing.  For each $n \in \en$ find a projection $r_n\leq q_{n+1}-q_n$ such that $\tau(r_n)\leq \frac{1}{2^n}$.  Then $r=\sum_{n\in\en} r_n\in \mc{K}_M$, and we have that for all $n \in \en$,
\[\norm{q_nr-r}=1.\]
This contradicts that $(q_n)_{n \in \en}$ is an approximate identity.
\item This follows from (3), since $\mathcal R\otimes \mathcal K(H)$ has a countable approximate identity and $\mathcal K_M\otimes \mathcal K(H)$ does not.  To see this, suppose that $(x_n)_{n \in \mathbb{N}}$ is a countable approximate identity for $\mathcal K_M\otimes \mathcal K(H)$, and let $p$ be a rank one projection in $\mathcal K(H)$.  Then $((1\otimes p)x_n(1\otimes p))_{n \in \mathbb{N}}$ is be a countable approximate identity for $\mathcal K_M\otimes p$, but $\mathcal K_M\otimes p\cong \mathcal K_M$, so this contradicts (3). 
\end{enumerate}
\end{proof}

There are many differences between the Calkin algebra and $M/\mathcal K_M$. Some of them are already clear from the $K$-theory considerations above, or from the fact that $\mathcal K(H)$ is separable.  Another difference, a little bit more subtle, is given by the following:
\begin{prop} 
Let $H$ be a separable Hilbert space, and let $Q$ be the canonical quotient map onto the Calkin algebra.  Let $(e_n)_{n \in \mathbb{N}}$ be an orthonormal basis for $H$, and let $S \in \mc{B}(H)$ be the unilateral shift in $\mathcal B(H)$ defined by $S(e_n)=e_{n+1}$ for all $n$. Then neither $S$ nor $Q(S)$ has a square root, but $1\otimes S\in \mathcal R\,\overline\otimes\,\mathcal B(H)$ does have a square root.
\end{prop}
\begin{proof}
Suppose that $Q(T) \in\mathcal C(H)$ is such that $Q(T)^2=Q(S)$.  Since $Q(S)$ is invertible in the Calkin algebra so is $Q(T)$.  The Fredholm index of $S$ is $-1$, so if $n \in \zet$ is the Fredholm index of $T$ then $2n=-1$, which is impossible.  Therefore $Q(S)$ has no square root, and hence neither does $S$.

For the second assertion recall that $\mathcal R\cong M_2(\mathcal R)$, and so
\[\mathcal R\,\overline\otimes\,\mc{B}(H)\cong M_2(\mathcal R\,\overline\otimes\,\mc{B}(H))=\mathcal R\,\overline \otimes\,(M_2\,\otimes\,\mc{B}(H)).\]
We view $\mc{B}(H)$ as embedded in $M_2 \otimes \mc{B}(H) = \mc{B}(H')$ for another Hilbert space $H'$; find $(f_n)_{n \in \mathbb{N}}$ such that $\set{e_n, f_n : n \in \mathbb{N}}$ is an orthonormal basis for $H'$.  Let $S' \in \mc{B}(H')$ be defined such that $S'(e_n) = f_n$ and $S'(f_n) = e_{n+1}$ for all $n$.  Then $T = 1 \otimes S' \in \mathcal R\,\overline{\otimes}\,\mc{B}(H')$, and $T^2 = 1 \otimes S$.
\end{proof}

A consequence of the previous proof, and of the fact that $\mathcal R\cong M_p(\mathcal R)$ for any integer $p$, is the following:
\begin{Cor}
$1\otimes S\in M$ has a $q^{\text{th}}$-root for every rational $q$.
\end{Cor}

With the motivating example in mind, we turn to establishing countable degree-$1$ saturation of a class of algebras containing $M / \mc{K}_M$.
\subsection*{A weakening of the $\sigma$-unital assumption.}
We recall the following result, which may be found in \cite[Corollary 6.3]{pedersencorona}:
\begin{Lemma}\label{lemma1}
Let $A$ be a C*-algebra, $S\in A_1$ and $T\in A_{\leq1}^+$. Then 
\[\norm{[S,T]}=\epsilon\leq \frac{1}{4} \Rightarrow \norm{[S,T^{1/2}]}\leq \frac{5}{4}\sqrt\epsilon.\]
\end{Lemma}

The following lemma is the key technical ingredient of Theorem \ref{thm:FactorCtbleSat} below.  It is a strengthening of the construction used in \cite[Lemma 3.4]{farah2011countable}, as if $A$ is $\sigma$-unital and $M = M(A)$ is the multiplier algebra of $A$, then $M$ and $A$ satisfy the hypothesis of our lemma.

\begin{Lemma}\label{thefns}
Let $M$ be a unital C*-algebra, let $A \subseteq M$ be an essential ideal, and let $\pi : M \to M/A$ be the quotient map.  Suppose that there is an increasing sequence $(g_n)_{n\in\en} \subset A$ of positive elements whose supremum is $1_M$, and suppose that any increasing uniformly bounded sequence converges in $M$.

Let $(F_n)_{n\in\en}$ be an increasing sequence of finite subsets of the unit ball of $M$ and $(\epsilon_n)_{n \in \en}$ be a decreasing sequence converging to $0$, with $\epsilon_0<1/4$.  Then there is an increasing sequence $(e_n)_{n \in \en} \subset A_{\leq 1}^+$ such that, for all $n\in\en$ and $a\in F_n$, the following conditions hold, where $f_n=(e_{n+1}-e_n)^{1/2}$:
\begin{enumerate}[(1)]
\item\label{cond0} $\abs{\norm{(1-e_{n-2})a(1-e_{n-2})}-\norm{\pi(a)}}<\epsilon_n$ for all $n\geq2$,
\item\label{cond1} $\norm{[f_n,a]}<\epsilon_n$ for all $n$,
\item\label{cond2a} $\norm{f_n(1-e_{n-2}) - f_n} < \epsilon_n$ for all $n \geq 2$,
\item\label{cond2} $\norm{f_nf_m}<\epsilon_m$ for all $m\geq n+2$,
\item\label{cond3} $\norm{[f_n,f_{n+1}]}<\epsilon_{n+1}$ for all $n$,
\item\label{cond4}  $\norm{f_naf_n}\geq \norm{\pi(a)}-\epsilon_n$ for all $n$,
\item\label{cond5}$\sum_{n \in \en} f_n^2=1$;
\listpart{
and further, whenever $(x_n)_{n \in \en}$ is a bounded sequence from $M$, the following conditions also hold:}
\item\label{cond6}the series $\sum_{n \in \en} f_nx_nf_n$ converges to an element of $M$,
\item\label{cond7}
\[\norm{\sum_{n \in \en} f_nx_nf_n}\leq \sup_{n \in \en}\norm{x_n},\]
\item\label{cond8} whenever $\limsup_{n \to \infty}\norm{x_n}=\limsup_{n\to\infty}\norm{x_nf_n^2}$ we have
\[\limsup_{n \to \infty}\norm{x_nf_n^2}\leq\norm{\pi\left(\sum_{n \in \en} x_nf_n^2\right)}.\]
\end{enumerate}
\end{Lemma}
\begin{proof}
For each $n \in \en$ let $\delta_n=10^{-100}\epsilon_n^2$, and let $(g_n)_{n \in \en}$ be an increasing sequence in $A$ whose weak limit is $1$.  We will build a sequence $(e_n)_{n \in \en}$ satisfying the following conditions:
\begin{enumerate}[(a)]
\item\label{conda} $\abs{\norm{(1-e_{n-2})a(1-e_{n-2})}-\norm{\pi(a)}}<\epsilon_n$ for all $n\geq2$ and $a\in F_n$,
\item\label{condb} $0\leq e_0\leq\ldots\leq e_n\leq e_{n+1}\leq\ldots\leq 1$,  and for all $n$ we have $e_n\in A$,
\item\label{condc} $\norm{e_{n}e_k-e_k}<\delta_{n+1}$ for all $n>k$,
\item \label{condd} $\norm{[e_n,a]}<\delta_n$ for all $n \in \en$ and $a\in F_{n+1}$,
\item\label{conde} $\norm{(e_{n+1}-e_n)a}\geq \norm{\pi(a)}-\delta_{n}$ for all $n \in \en$ and $a\in F_{n}$
\item\label{condf} $\norm{(e_{m+1}-e_m)^{1/2}e_n(e_{m+1}-e_m)^{1/2}-(e_{m+1}-e_m)}<\delta_{n+1}$ for all $n>m+1$,
\item\label{condg} $e_{n+1}\geq  g_{n+1}$ for all $n \in \en$.
\end{enumerate}

We claim that such a sequence will satisfy \ref{cond0}--\ref{cond5}, in light of Lemma \ref{lemma1}.  Conditions \ref{cond0} and \ref{conda} are identical.  Condition \ref{condd} implies condition \ref{cond1}.  Condition \ref{condc} and the C*-identity imply condition \ref{cond2a}, which in turn implies conditions \ref{cond2} and \ref{cond3}.  We have also that conditions \ref{conde} and \ref{condg} imply respectively conditions \ref{cond4} and \ref{cond5}, so the claim is proved.  After the construction we will show that \ref{cond6}--\ref{cond8} also hold.

Take $\Lambda=\{\lambda\in A^+\colon \lambda\leq 1\}$ to be the approximate identity of positive contractions (indexed by itself) and let $\Lambda'$ be a subnet of $\Lambda$ that is quasicentral for $M$ (see \cite[Theorem 2.1]{pedersencorona} or  \cite[$\S$1]{ArvesonNotes}). 

Since $A$ is an essential ideal of $M$, by  \cite[II.6.1.6]{BlackadarOpAlg} there is a faithful representation $\beta$ on an Hilbert space $H$ such that \[1_{H}=\text{SOT}-\lim_{\lambda\in\Lambda'}\{\beta(\lambda)\}, \] 
Consequently, for every finite $F\subset M$, $\epsilon>0$ and $\lambda\in\Lambda'$ there is $\mu>\lambda$ such that for all $a \in F$,
\[\nu\geq \mu\Rightarrow \norm{(\nu-\lambda)a}\geq \norm{\pi(a)}-\epsilon.\]

We will proceed by induction. Let $e_{-1}=0$ and $\lambda_0\in\Lambda'$ be such that for all $\mu>\lambda_0$ and $a\in F_1$ we have $\norm{[\mu,a]}<\delta_0$. By cofinality of $\Lambda'$ in $\Lambda$ we can find a $e_0\in\Lambda'$ such that $e_0>\lambda_0, g_0$. Find now $\lambda_1>e_0$ such that for all $\mu>\lambda_1$ and $a\in F_2$ we have \[\norm{[\mu,a]}<\delta_1, \, \norm{(\mu-e_0)a}\geq \norm{\pi(a)}-\delta_1.\] Since we have that

\begin{equation}\label{limofnets}\norm{\pi(a)}=\lim_{\lambda\in\Lambda'}\norm{(1-\lambda)a(1-\lambda)}
\end{equation} we can also ensure that for all $a\in F_{3}$ and all $\mu>\lambda_{1}$, condition \ref{cond0} is satisfied.

Picking $e_1\in\Lambda'$ such that $e_1>\lambda_1,g_1$ we have that the base step is completed.

Suppose now that $e_0,\ldots,e_{n}, f_0,\ldots,f_{n-1}$ are constructed. 

We can choose $\lambda_{n+1}$ so that for all $\mu>\lambda_{n+1}$, with $\mu\in \Lambda'$, we have $\norm{[\mu,a]}<\delta_{n+1}/4$ and $\norm{(\mu-e_n)a}\geq \norm{\pi(a)}-\delta_n$  for $a\in F_{n+2}$. Moreover, by the fact that $\Lambda'$ is an approximate identity for $A$ we can have that $\norm{f_m\mu f_m-f_m^2}<\delta_{n+2}$ for every $m<n$ and that $\norm{\mu e_k-e_k}<\delta_{n+2}$ for all $k\leq n$. By equation (\ref{limofnets}) we can also ensure that for all $a\in F_{n+2}$ and all $\mu>\lambda_{n+1}$, condition \ref{cond0} is satisfied.

Once this $\lambda_{n+1}$ is picked we may choose \[e_{n+1}\in \Lambda',\,\,\,e_{n+1}> \lambda_{n+1}, \, g_{n+1},\] to end the induction.  

It is immediate from the construction that the sequence $(e_n)_{n \in \en}$ chosen in this way satisfies conditions \ref{conda} - \ref{condg}.  To complete the proof of the lemma we need to show that conditions \ref{cond6}, \ref{cond7} and \ref{cond8} are satisfied by the sequence $\{f_n\}$.

To prove \ref{cond6}, we may assume without loss of generality that each $x_n$ is a contraction.  Recall that every contraction in $M$ is a linear combination (with complex coefficients of norm $1$) of four positive elements of norm less than $1$, and addition and multiplication by scalar are weak operator continuous functions.  It is therefore sufficient to consider a sequence $(x_n)$ of positive contractions.
By positivity of $x_n$, we have that $(\sum_{i\leq n} f_ix_if_i)_{n \in \en}$ is an increasing uniformly bounded sequence, since for every $n$ we have 
\[\sum_{i\leq n} f_ix_if_i\leq\sum_{i\leq n}f_i^2\qquad\text{ and }\qquad f_nx_nf_n\geq 0.\] 
Hence $(\sum_{i\leq n} f_ix_if_i)_{n \in \en}$ converges in weak operator topology to an element of $M$  of bounded norm, namely the supremum of the sequence, which is $\sum_{n\in\en} f_nx_nf_n$.

For \ref{cond7}, consider the algebra $\prod_{k\in\en} M$ with the sup norm and the map $\phi_n\colon \prod_{k\in\en} M\to M$ such that $\phi_n((x_i))=f_nx_nf_n$. Each $\phi_n$ is completely positive, and since $f_n^2\leq\sum_{i\in\en} f_i^2=1$, also contractive.  For the same reason the maps $\psi_n\colon\prod_{k\in\en} M\to M$ defined as $\psi_n((x_i))=\sum_{j\leq n}f_jx_jf_j$ are completely positive and contractive. Take $\Psi$ to be the supremum of the maps $\psi_n$.  Then $\Psi((x_n))=\sum_{i\in\en} f_ix_if_i$. This map is a completely positive map of norm $1$, because $\norm{\Psi}=\norm{\Psi(1)}$, and from this condition \ref{cond7} follows.

For \ref{cond8}, we can suppose $\limsup_{i\to\infty}\norm{x_i}=\limsup_{i\to\infty}\norm{x_if_i^2} = 1$. Then for all $\epsilon > 0$ there is a sufficiently large $m \in \en$ and a unit vector $\xi_m \in H$ such that 
\[\norm{x_mf_m^2(\xi_m)}\geq 1-\epsilon.\]
Since $\norm{x_i}\leq 1$ for all $i$, we have that $\norm{f_m(\xi_m)}\geq 1-\epsilon$, that is, $\abs{(f_m^2\xi_m\mid\xi_m)}\geq 1-\epsilon$. In particular we have that $\norm{\xi_m-f_m^2(\xi_m)}\leq\epsilon$. 

Since $\sum f_i^2=1$ we have that $\xi_m$ and $\xi_n$ constructed in this way are almost orthogonal for all $n,m$. In particular, choosing $\epsilon$ small enough at every step, we are able to construct a sequence of unit vectors $\{\xi_m\}$ such that $\abs{(\xi_m\mid\xi_n)}\leq 1/2^m$ for $m>n$. But this means that for any finite projection $P\in M$ only finitely many $\xi_m$ are in the range of $P$ up to $\epsilon$ for every $\epsilon>0$. In particular, if $I$ is the set of all convex combinations of finite projections, we have that that 
\[\lim_{\lambda\in I}\norm{\sum_{i\in\en} x_if_i^2-\lambda\left( \sum_{i\in\en} x_if_i^2\right)}\geq 1.\] 
Since $I$ is an approximate identity for $A$ we have that 
\[\norm{\pi\left(\sum_{i\in\en} x_if_i^2\right)}=\lim_{\lambda\in I}\norm{\sum_{i\in\en} x_if_i^2-\lambda\left(\sum_{i\in\en} x_if_i^2\right)},\] 
as desired.
\end{proof}
 We can then proceed with the proof of the main theorem.
\begin{teo}\label{thm:FactorCtbleSat}
Let $M$ be a unital C*-algebra, and let $A \subseteq M$ be an essential ideal.  Suppose that there is an increasing sequence $(g_n)_{n\in\en} \subset A$ of positive elements whose supremum is $1_M$, and suppose that any increasing uniformly bounded sequence converges in $M$.  Then $M/A$ is countably degree-$1$ saturated.
\end{teo}
\begin{proof}
Let $\pi : M \to M/A$ be the quotient map.  Let $(P_n(\overline x))_{n \in \en}$ be a collection of $^*$-polynomial of degree $1$ with coefficients in $M/A$, and for each $n \in \en$ let $r_n\in\er^+$.  Without loss of generality, reordering the polynomials and eventually adding redundancy if necessary, we can suppose that the only variables occurring in $P_n$ are $x_0, \ldots, x_n$.

Suppose that the set of conditions $\set{\norm{P_n(x_0,\ldots,x_n)} = r_n : n \in \en}$ is approximately finitely satisfiable, in the sense of Definition \ref{cd1s}.  As we noted immediately after Definition \ref{cd1s}, it is sufficient to assume that the partial solutions are all in $(M/A)_{\leq 1}$, and we must find a total solution also in $(M/A)_{\leq 1}$.  So we have partial solutions
\[\{\pi(x_{k,i})\}_{k\leq i}\subseteq (M/A)_{\leq1}\]
such that for all $i\in\en$ and $n\leq i$ we have
\[\norm{P_n(\pi(x_{0,i}),\ldots,\pi(x_{n,i}))}\in (r_n)_{1/i}.\] 

For each $n \in \en$, let $Q_n(x_0,\ldots,x_n)$ be a polynomial whose coefficients are liftings of the coefficients of $P_n$ to $M$, and let $F_n$ be a finite set that contains
\begin{itemize}
\item all the coefficients of $Q_k$, for $k\leq n$
\item $x_{k,i}, x_{k,i}^*$ for $k\leq i\leq n$.
\item $Q_k(x_{0,i},\ldots,x_{k,i})$ for $k\leq i\leq n$.
\end{itemize}
Let $\epsilon_n = 4^{-n}$.  Find sequences $(e_n)_{n \in \en}$ and $(f_n)_{n \in \en}$ satisfying the conclusion of Lemma \ref{thefns} for these choices of $(F_n)_{n \in \en}$ and $(\epsilon_n)_{n \in \en}$.

Let $\overline x_{n,i}=(x_{0,i},\ldots,x_{n,i})$, $y_k=\sum_{i\geq k} f_ix_{k,i}f_i$, $\overline y_n=(y_0,\ldots,y_n)$ and $\overline z_n=\pi(\overline y_n)$.  Fix $n \in \en$; we will prove that $\norm{P_n(\overline z_n)}=r_n$.

First, since $x_{k,i}\in M_{\leq 1}$, as a consequence of condition \ref{cond7} of Lemma \ref{thefns}, we have that $y_i\in M_{\leq 1}$   for all $i$. 
 Moreover, since $Q_n$ is a polynomial whose coefficients are lifting of those of $P_n$ we have 
\[\norm{P_n(\overline z_n)}=\norm{\pi(Q_n(\overline y_n))}.\]
We claim that 
\[Q_n(\overline y_n)-\sum_{j \in \en} f_jQ_n(\overline x_{n,j})f_j\in A.\] 
It is enough to show that 
\[\sum_{j \in \en} f_jax_{k,j}bf_j-\sum_{j \in \en} af_jx_{k,j}f_jb\in A,\] 
where $a,b$ are coefficients of a monomial in $Q_n$,  since $Q_n$ is the sum of finitely many of these elements (and the proof for monomials of the form $ax_{k,j}^*b$ is essentially the same as the one for $ax_{k, j}b$). 

By construction we have $a,b\in F_n$, and hence by condition \ref{cond1} of Lemma \ref{thefns}, for $j$ sufficiently large, 
\[\forall x\in M_{\leq 1}\,(\norm{af_jxf_jb-f_jaxbf_j}\leq 2^{-j}(\norm{a}+\norm{b})).\] 
Therefore $\sum_{j \in \en} (f_jax_{k,j}bf_j-af_jx_{k,j}f_jb)$ is a series of elements in $A$ that is converging in norm, which implies that the claim is satisfied. In particular, 
\[\norm{P_n(\overline z_n)}=\norm{\pi\left(\sum_{j \in \en} f_jQ_n(\overline x_{n,j})f_j\right)}.\]
For each $j \geq 2$, let $a_{j}=(1-e_{j-2})Q_n(\overline x_{n,j})(1-e_{j-2})$. By condition \ref{cond0} of Lemma \ref{lemma1}, the fact that $Q_n(\overline x_{n,j})\in F_n$, and the original choice of the $x_{n, j}$'s, we have that $\limsup \norm{a_j}=r_n$.
Similarly to the above, but this time using condition \ref{cond2a} of Lemma \ref{lemma1}, we have
\[\norm{\pi\left(\sum_{j \in \en} f_jQ_n(\overline x_{n,j})f_j\right)}=\norm{\pi\left(\sum_{j \in \en} f_ja_jf_j\right)}\leq \norm{ \sum_{j \in \en} f_ja_jf_j}.\]

Using condition \ref{cond7} of Lemma \ref{thefns} and the fact that $Q_n(\overline x_{n,j})\in F_j$ we have that \[\norm{\sum_{j \in \en} f_ja_jf_j}\leq \limsup_{j\to\infty} \norm{a_j}=r_n.\]
Combining the calculations so far, we have shown
\[\norm{P_n(\overline z_n)} = \norm{\pi\left(\sum_{j\in\en}f_jQ_n(\overline{x}_{n, j})f_j\right)} = \norm{\pi\left(\sum_{j \in \en}f_ja_jf_j\right)} \leq r_n.\]
Since $Q_n(\overline{x}_{n, j}) \in F_j$ for all $j$, condition \ref{cond4} of Lemma \ref{lemma1} implies
\[r_n \leq \limsup_{j\to\infty}\norm{f_jQ_n(\overline{x}_{n,j})f_j}.\]
It now remains to prove that
\[\limsup_{j\to\infty} \norm{f_ja_jf_j}\leq\norm{\pi\left(\sum_{j\in\en} f_ja_jf_j\right)}\]
so that we will have
\begin{align*}
r_n &\leq \limsup_{j\to\infty} \norm{ f_jQ_n(\overline x_{n,j})f_j} \\
 &=\limsup_{j\to\infty} \norm{f_ja_jf_j}\\
 &\leq \norm{\pi\left(\sum_{j\in\en} f_ja_jf_j\right)} \\
 &= \norm{P_n(\overline{z_n})}.
\end{align*}
We have $Q_n(\overline x_{n,j})\in F_j$, so by condition \ref{cond1} of Lemma \ref{lemma1}, we have that 
\[\limsup_{j\to\infty} \norm{f_ja_jf_j}=\limsup_{j\to\infty} \norm{a_jf_j^2},\]
and hence
\[\sum_{j\in\en} f_ja_jf_j-\sum_{j\in\en} a_jf_j^2\in A.\]
The final required claim will then follow by condition \ref{cond8} of Lemma \ref{lemma1}, once we verify \[\limsup_{j\to\infty}\norm{a_jf_j^2} = \limsup_{j\to\infty}\norm{a_j}.\]
We clearly have that for all $j$,
\[\norm{a_jf_j^2} \leq \norm{a_j}.\]
On the other hand, 
\begin{align*}
\limsup_{j \to \infty}\norm{a_jf_j^2} &= \limsup_{j\to\infty}\norm{f_ja_jf_j} \\
 &= \limsup_{j\to\infty}\norm{f_jQ_n(\overline{x}_{n, j})f_j} &\text{by condition \ref{cond2a}} \\
 &\geq r_n \\
 &= \limsup_{j\to\infty} \norm{a_j}.
\end{align*}
\end{proof}

The theorem above applies, in particular, to coronas of $\sigma$-unital algebras.  The following result is due to Farah and Hart, but unfortunately their proof in \cite{farah2011countable} has a technical error.  Specifically, our proof of Theorem \ref{thm:FactorCtbleSat} uses the same strategy as in \cite[Theorem 1.4]{farah2011countable}, but avoids their equation (10) on p. 14, which is incorrect.

\begin{Cor}[{\cite[Theorem 1.4]{farah2011countable}}]\label{teo:farahhartsigmaunital}
If $A$ is a $\sigma$-unital C*-algebra, then its corona $C(A)$ is countably degree-$1$ saturated.
\end{Cor}

We also obtain countable degree-$1$ saturation for the motivating example from the beginning of this section.

\begin{Cor}
Let $N$ be a II$_1$ factor, $H$ a separable Hilbert space and $M=N\,\overline\otimes\,\mathcal B(H)$ be the associated II$_\infty$ factor. Let $\mathcal K_M$ be the unique two-sided closed ideal of $M$, that is the closure of the elements of finite trace.  Then $M/\mc{K}_M$ is countably degree-$1$ saturated.  In particular, this is the case when $N = \mc{R}$, the hyperfinite II$_{1}$ factor.
\end{Cor}
\noindent
More generally, recall that a von Neumann algebra $M$ is \emph{finite} if there is not a projection that is Murray-von Neumann equivalent to $1_M$, and $\sigma$-\emph{finite} if there is a sequence of finite projections weakly converging to $1_M$.
\begin{Cor}
Let $M$ be a $\sigma$-finite but not finite tracial von Neumann algebra, and let $A$ be the ideal generated by the finite trace projections. Then $M/A$ is countably degree-$1$ saturated.
\end{Cor}

\section{Generalized Calkin algebras}\label{section:Calkin2}
\begin{Notation}Let $\alpha$ be an ordinal and $H_\alpha = \ell^2(\aleph_\alpha)$ be the unique (up to isomorphism) Hilbert space of density character $\aleph_\alpha$. Let $\mathcal B_\alpha=\mathcal B(H_{\alpha})$. Let $\mathcal K_\alpha$ be the ideal of compact operators in $\mathcal B_\alpha$. The quotient $\mathcal C_\alpha=\mathcal B_\alpha/\mathcal K_\alpha$ is called the \defined{generalized Calkin algebra of weight $\aleph_\alpha$}.
\end{Notation}

Note that, when $H$ is separable, the ideal of compact operators in $\mathcal B(H)$ is separable, and in particular $\sigma$-unital, so it follows from Corollary \ref{teo:farahhartsigmaunital} that the Calkin algebra is countably degree-$1$ saturated.  

We are going to give explicit results on the theories of the generalized Calkin algebras. Is it known (see \cite{farah2011countable}) that the Calkin algebra is not countably quantifier-free saturated; we show that the generalized Calkin algebras also fail to have this degree of saturation.  This follows immediately from the fact that the Calkin algebra is isomorphic to a corner of the generalized Calkin algebra and that if $A$ is a C*-algebra that is $\Phi$-saturated, where $\Phi$ include all $^*$-polynomials of degree $1$, then every corner of $A$ is $\Phi$-saturated. On the other hand, the proof shown below is direct and much easier than the proof in the separable case.  It is worth noting, however, that the method we will use does not apply to the Calkin algebra $\mathcal C_{0}$ itself.
\begin{Lemma}
Let $\alpha \geq 1$ be an ordinal. Then $\mathcal C_\alpha$ is not countably quantifier-free saturated.
\end{Lemma}
\begin{proof}
Fix $\{A_n\}_{n\in\en}$ a countable partition of $\aleph_\alpha$ in disjoint pieces of size $\aleph_\alpha$ and a base $(e_\beta)_{\beta < \aleph_\alpha}$ for $H_{\aleph_\alpha}$.  For each $n \in \en$ let $P_n$ be the projection onto $\overline{\spann(e_\beta\colon\beta\in A_n)}$.

\begin{claim} If $Q$ is a projection in $\mathcal B_\alpha$ such that $QP_n\in\mathcal K_\alpha$ for all $n$ then $Q$ has range of countable density.
\end{claim}
\begin{proof}
We have that for any $n\in\en$ and $\epsilon>0$ there is a finite $C_{\epsilon,n}\subseteq\aleph_\alpha$ such that
\[\beta\notin C_{\epsilon,n}\Rightarrow \norm{QP_ne_\beta}<\epsilon.\]

Let $D=\bigcup_{n\in\en}\bigcup_{m\in\en} C_{1/m,n}$. If $\beta\notin D$ then for all $n\in\en$ we have $\norm{QP_ne_\beta}=0$ and since there is $n$ such that $e_\beta\in P_n$, we have that $\norm{Qe_\beta}=0$. Since $D$ is countable, $Q$ is identically zero on a subspace of countable codimension.
\end{proof}
Let $Q_{-4}=xx^*-1$, $Q_{-3}=x^*x-y$,  $Q_{-2}=y-y^*$, $Q_{-1}=y-y^2$, and $Q_n=yP_n$. The type $\{\norm{Q_{i}} = 0\}_{-4\leq i}$ admits a partial solution, but not a total solution.
\end{proof}
We are going to have a further look at the theories of $\mathcal C_\alpha$. In particular we want to see if it is possible to distinguish between the theories of $\mathcal C_\alpha$ and of $\mathcal C_\beta$, whenever $\alpha\neq\beta$. Of course, since there are at most $2^{\aleph_0}$ many possible theories, we have that there are ordinals $\alpha\neq\beta$ such that $\mathcal C_\alpha\equiv \mathcal C_\beta$.  As we show in the next theorem, this phenomenon cannot occur whenever $\alpha$ and $\beta$ are sufficiently small, and similarly for $\mc{B}_\alpha$ and $\mc{B}_\beta$.

\begin{teo}\label{teo:generalizedcalkin} 
Let $\alpha \neq \beta$ be ordinals, and $H_\alpha$ the Hilbert space of density $\aleph_\alpha$. Then the projections of the algebras $\mathcal C_\alpha$ and $\mc{C}_\beta$ as posets with respect to the Murray-von Neumann order are elementary equivalent if and only if $\alpha=\beta\mod\omega^\omega$,  where $\omega^\omega$ is computed by ordinal exponentiation, as they are the infinite projections of $\mathcal B_\alpha$ of $\mc{B}_\beta$. Consequently, if $\alpha\not\equiv\beta$ then $\mathcal B_\alpha\not\equiv \mathcal B_\beta$ and $\mathcal C_\alpha\not\equiv\mathcal C_\beta$.
\end{teo}
\begin{proof}
The key fact is that $\alpha \equiv \beta$ (as first-order structures with only the ordering) if and only if $\alpha = \beta \mod\omega^{\omega}$; see \cite[Corollary 44]{elemequiv}.  Hence the proof will be complete as soon as we notice that the ordinal $\alpha$ is interpretable in both $\mc{C}_\alpha$ (as the set of projections under Murray-von Neumann equivalence) and inside $\mc{B}_\alpha$ (as the set of infinite projections under Murray-von Neumann equivalence).

To see this note that there is a formula $\phi$ such that $\phi(p,q)=0$ if $p\sim_{MvN}q$ and $p,q$ are projections and $\phi(p,q)=1$ otherwise, and that being an infinite projection is axiomatizable (since $p$ is an infinite projection if and only if $\psi(p)=0$ if and only if $\psi(p)<1/4$, where 
\[\psi(x)=\norm{x-x^*}+\norm{x-x^2}+\inf_y\left(\norm{yy^*-x}+\norm{y^*yx-y^*y} + (1 \dot{-} \norm{y^*y-x})\right)\] where $y$ ranges over the set of partial isometries. Since we have that to any projection we can associate the density of its range (both in $\mathcal C_\alpha$ and $\mathcal B_\alpha$), and that we have that $p\leq_{MvN}q$ if and only if the density of $p$ is less or equal than the range of $q$. Since every possible value for the density is of the form $\aleph_\beta$, for $\beta<\alpha$, the theorem is proved.
\end{proof}

\section{Abelian C*-algebras}\label{section:Topology}
In this section we consider abelian C*-algebras, and particularly the theories of real rank zero abelian C*-algebras. In the first part of this section we give a full classification of the complete theories of abelian real rank zero C*-algebras in terms of the (discrete first-order) theories of Boolean algebras (recall that a theory is \emph{complete} if whenever $M \models T$ and $N \models T$ then $M\equiv N$).  As an immediate consequence of this classification we find that there are exactly $\aleph_0$ distinct complete theories of abelian real rank zero C*-algebras.  We also give a concrete description of two of these complete theories.

In the second part of the section we return to studying saturation.  We show how saturation of abelian C*-algebras is related to the classical notion of saturation for Boolean algebras.  We begin by recalling some well-known definitions and properties.

\subsection{Preliminaries from topology and Boolean algebra}
\begin{Notation} A topological space $X$ is said \defined{sub-Stonean} if any pair of disjoint open $\sigma$-compact sets has disjoint closures; if, in addition, those closures are open and compact, $X$ is said \defined{Rickart}. A space $X$ is said to be \defined{totally disconnected} if the only connected components of $X$ are singletons and \defined{$0$-dimensional} if $X$ admits a basis of clopen sets.

A topological space $X$ such that every collection of disjoint nonempty open subsets of $X$ is countable is said to carry the \defined{countable chain condition}.
\end{Notation}  

Note that for a compact space being totally disconnected is the same as being $0$-dimensional, and this corresponds to the fact that $C(X)$ has real rank zero. Moreover any compact Rickart space is $0$-dimensional and sub-Stonean, while the converse is false (take for example $\beta\en\setminus\en$).  The space $X$ carries the countable chain condition if and only if $C(X)$ has few orthogonal positive elements (see Definition \ref{def:sccc}).

\begin{remark}
Let $X$ be a compact $0$-dimensional space, $CL(X)$ its algebra of clopen sets. For a Boolean algebra $B$, let $S(B)$ its Stone space, i.e., the space of all ultrafilters in $B$.

Note that if two $0$-dimensional spaces $X$ and $Y$ are homeomorphic then $CL(X)\cong CL(Y)$ and conversely, we have that $CL(X)\cong CL(Y)$ implies $X$ and $Y$ are homeomorphic to $S(CL(X))$.

Moreover, if $f\colon X\to Y$ is a continuous map of compact $0$-dimensional spaces we have that $\phi_f\colon CL(Y)\to CL(X)$ defined as $\phi_f(C)=f^{-1}[C]$ is an homomorphism of Boolean algebras.  Conversely, for any homomorphism of Boolean algebras $\phi\colon CL(Y)\to CL(X)$ we can define a continuous map $f_\phi\colon X\to Y$. If $f$ is injective, $\phi_f$ is surjective. If $f$ is onto $\phi_f$ is $1$-to-$1$ and same relations hold for $\phi$ and $f_\phi$.
\end{remark}

We recall some basic definitions and facts about Boolean algebras.
\begin{Notation}
A Boolean algebra is atomless if $\forall a\neq 0$ there is $b$ such that $0<b<a$. For $Y,Z\subset B$ we say that $Y<Z$ if $\forall (y,z)\in Y\times Z$ we have $y< z$.
\end{Notation}
Note that, for a $0$-dimensional space, $CL(X)$ is atomless if and only if $X$ does not have isolated points. In particular \[\abs{\{a\in CL(X)\colon a\text{ is an atom}\}}=\abs{\{x\in X\colon x\text{ is isolated}\}}.\]
\begin{defin}
Let $\kappa$ be an uncountable cardinal. A Boolean algebra $B$ is said to be \defined{$\kappa$-saturated} if every finitely satisfiable type of cardinality $<\kappa$ in the first-order language of Boolean algebras is satisfiable.
\end{defin}

For atomless Boolean algebras this model-theoretic saturation can be equivalently rephrased in terms of increasing and decreasing chains:

\begin{teo}[{\cite[Thm 2.7]{Mija79}}]\label{teo:satbooleanalgebra}
Let $B$ be an atomless Boolean algebra, and $\kappa$ an uncountable cardinal.  Then $B$ is $\kappa$-saturated if and only if for every directed $Y<Z$ such that $\abs{Y}+\abs{Z}<\kappa$ there is $c\in B$ such that $Y<c<Z$.
\end{teo}

\subsection{Elementary equivalence}
\begin{Notation}
Let $\mc{U}$ be an ultrafilter (over a possibly uncountable index set).  If $A$ is a C*-algebra, we denote the C*-algebraic ultrapower of $A$ by $\mc{U}$ by $A^{\mc{U}}$.  Similarly, if $B$ is a Boolean algebra we denote the classical model-theoretic ultrapower of $B$ by $B^{\mc{U}}$.  If $X$ is a topological space we denote the ultracopower by $\sum_{\mc{U}}X$.
\end{Notation}

The reader unfamiliar with the ultracopower construction is referred to \cite{Bankston}.  The only use we will make of this tool is the following lemma.

\begin{Lemma}[{\cite[Proposition 2]{Gurevic} and \cite[Remark 2.5.1]{Bankston}}]\label{lem:Bankston}
Let $X$ be a compact Hausdorff space, and let $\mc{U}$ be an ultrafilter.  Then $C(X)^{\mc{U}} \cong C(\sum_{\mc{U}}X)$, and $CL(X)^{\mc{U}} \cong CL(\sum_{\mc{U}}X)$.
\end{Lemma}

\begin{teo}\label{boolean0dim}
Let $A$ and $B$ be abelian, unital, real rank zero C*-algebras.  Write $A = C(X)$ and $B = C(Y)$, where $X$ and $Y$ are $0$-dimensional compact Hausdorff spaces.  Then $A \equiv B$ as metric structures if and only if $CL(X) \equiv CL(Y)$ as Boolean algebras.
\end{teo}
\begin{proof}
Suppose that $A \equiv B$.  By the Keisler-Shelah theorem (Theorem \ref{thm:KeislerShelah}) there is an ultrafilter $\mc{U}$ such that $A^{\mc{U}} \cong B^{\mc{U}}$.  By Lemma \ref{lem:Bankston} $A^{\mc{U}} \cong C(\sum_{\mc{U}}X)$.  Thus we have $C(X_{\mc{U}}) \cong C(Y_{\mc{U}})$, and hence by Gelfand-Naimark $X_{\mc{U}}$ is homeomorphic to $Y_{\mc{U}}$.  Then $CL(\sum_{\mc{U}}X) \cong CL(\sum_{\mc{U}}Y)$.  Applying Lemma \ref{lem:Bankston} again, we have $CL(\sum_{\mc{U}}X) = CL(X)^{\mc{U}}$, so we obtain $CL(X)^{\mc{U}} \cong CL(Y)^{\mc{U}}$, and in particular, $CL(X) \equiv CL(Y)$.  The converse direction is similar, starting from the Keisler-Shelah theorem for first-order logic (see \cite{Shelah}).
\end{proof}

It is interesting to note that the above result fails when $C(X)$ is considered only as a ring in first-order discrete logic (see \cite[Section 2]{Bankston2}).

\begin{Cor}
There are exactly $\aleph_0$ distinct complete theories of abelian, unital, real rank zero C*-algebras.
\end{Cor}
\begin{proof}
There are exactly $\aleph_0$ distinct complete theories of Boolean algebras; see \cite[Theorem 5.5.10]{ChangKeisler} for a description of these theories.
\end{proof}
\begin{Cor}\label{cor:elementaryEquivalence}
If $X$ and $Y$ are infinite, compact, $0$-dimensional spaces both with the same finite number of isolated points or both having a dense set of isolated points, then $C(X)\equiv C(Y)$.

In particular, let $\alpha$ be any infinite ordinal. Then $C(\alpha+1)\equiv C(\beta\omega)$. Moreover, if $\alpha$ is a countable limit, $C(2^\omega)\equiv C(\beta\omega\setminus\omega)\equiv C(\beta\alpha\setminus\alpha)$.
\end{Cor}
\begin{proof}Given $X,Y$ as in the hypothesis, again by Theorem  \cite[Theorem 5.5.10]{ChangKeisler}, we have that $CL(X)\equiv CL(Y)$.
\end{proof}
\begin{remark}
The construction of $C(X)$ from the Boolean algebra $CL(X)$ can be described directly.  Recall that $C(S(B))=\overline{\ce B}$ where 
\[\ce B=\{\sum_{i\leq n}\lambda_ix_i\mid \lambda_i\in\ce, x_i\in B, n\in\en\}.\]
In the same way, given an abelian group $\Gamma$, we can define the associated C*-algebra $\overline{\ce\Gamma}$.  Martino Lupini asked whether the analogue of Theorem \ref{boolean0dim} holds in this setting, that is, whether it is true that $\Gamma\equiv\Gamma'\iff\overline{\ce\Gamma}\equiv\overline{\ce\Gamma'}$ for countable abelian groups.

In fact, both implications fail.  For one direction, recall that $\overline{\ce\Gamma}=C(\hat\Gamma)$ where $\hat\Gamma$ is the dual group of $G$.  If $p$ is prime then the dual of $\bigoplus_{\en}\zet/p\zet$ is $p^\en$, hence \[\overline{\ce\bigoplus_\en\zet/p\zet}\cong\overline{\ce\bigoplus_\en\zet/q\zet}\cong C(2^\en)\] for all primes $p$ and $q$; clearly the groups are not elementary equivalent.

For the forward implication, we given an example pointed out to us by Tomasz Kania.  It is known that any two torsion-free divisible abelian groups are elementarily equivalent (see \cite[p. 40]{ChangKeisler}), so in particular, $\qu \cong \qu \oplus \qu$.  The dual group of $\qu$ with the discrete topology is a $1$-dimensional indecomposable continuum (see \cite[25.4, p. 404]{HewittRoss}), but the dual group of $\qu\oplus\qu$ is $2$-dimensional.  Hence $\overline{\ce \qu} \not\equiv \overline{\ce(\qu \oplus \qu)}$.
\end{remark}
\subsection{Saturation}
This section is dedicated to the analysis of the relations between topology and countable saturation of abelian C*-algebras. In particular, we want to study which kind of topological properties the compact Hausdorff space $X$ has to carry in order to have some degree of saturation of the metric structure $C(X)$ and, conversely, to establish properties that are incompatible with the weakest degree of saturation of the corresponding algebra.  From now on $X$ will denote an infinite compact Hausdorff space (note that if $X$ is finite then $C(X)_{\leq 1}$ is compact, and so $C(X)$ is fully saturated).

The first limiting condition for the weakest degree of saturation are given by the following Lemma:
\begin{Lemma}\label{abcondition1}
Let $X$ be an infinite compact Hausdorff space, and suppose that $X$ satisfies one of the following conditions:
\begin{enumerate}[(1)]
\item\label{abcond1} $X$ has the countable chain condition;
\item\label{abcond2} $X$ is separable;
\item\label{abcond3} $X$ is metrizable;
\item\label{abcond3a} $X$ is homeomorphic to a product of two infinite compact Hausdorff spaces;
\item\label{abcond4} $X$ is not sub-Stonean;
\item\label{abcond5} $X$ is Rickart.
\end{enumerate}
Then $C(X)$ is not countably degree-$1$ saturated.
\end{Lemma}
\begin{proof}
First, note that \ref{abcond3} $\Rightarrow$ \ref{abcond2} $\Rightarrow$ \ref{abcond1}.  The fact that \ref{abcond1} implies that $C(X)$ is not countably degree-$1$ saturated is an instance of Lemma \ref{lem:scccbad}.  Failure of countable degree-$1$ saturation for spaces satisfying \ref{abcond3a} follows from Theorem \ref{thm:ConsequencesOfSat}, while for those satisfying \ref{abcond4} it follows from \cite[Remark 7.3]{pedersencorona} and \cite[Proposition 2.6]{farah2011countable}.  It remains to consider \ref{abcond5}.

Let $X$ be Rickart.  The Rickart condition can be rephrased as saying that any bounded increasing monotone sequence of self-adjoint functions in $C(X)$ has a least upper bound in $C(X)$ (see \cite[Theorem 2.1]{pedersensubstonean}). 

Consider a sequence $(a_n)_{n\in\en}\subseteq C(X)_1^+$ of positive pairwise orthogonal elements, and let $b_n=\sum_{i\leq n} a_i$.  Then $(b_n)_{n\in\en}$ is a bounded increasing sequence of positive operators,  so it has a least upper bound $b$. Since $\norm{b_n}=1$ for all $n$, we also have $\norm{b}=1$. The type consisting of $P_{-3}(x)=x$, with $K_{-3}=\{1\}$, $P_{-2}(x)=b-x$ with $K_{-2}=[1,2]$, $P_{-1}(x)=b-x-1$ with $K_{-1}=\{1\}$ and $P_n(x)=x-b_n-1$ with $K_n=[0,1]$ is consistent with partial solution $b_{n+1}$ (for $\{P_{-3},\ldots,P_n\}$).  This type cannot have a positive solution $y$, since in that case we would have that $y-b_n\geq 0$ for all $n\in\en$, yet $b-y>0$, a contradiction to $X$ being Rickart.
\end{proof}

Note that the preceding proof shows that the existence of a particular increasing bounded sequence that is not norm-convergent but does have a least upper bound (a condition much weaker than being Rickart) is sufficient to prove that $C(X)$ does not have countable degree-$1$ saturation. Moreover, the latter argument does not use that the ambient algebra is abelian.

We will compare the saturation of $C(X)$ (in the sense of Definition \ref{cd1s}) with the saturation of $CL(X)$, in the sense of the above theorem.

The results that we are going to obtain are the following:
\begin{teo}\label{theorem1}
Let $X$ be a compact $0$-dimensional Hausdorff space. Then 
\[C(X) \text{ is countably saturated} \Rightarrow CL(X)\text{ is countably saturated }\] 
and 
\[CL(X)\text{ is countably saturated }\Rightarrow C(X) \text{ is countably q.f. saturated.}\]
\end{teo}
\begin{teo}\label{theorem2}
Let $X$ be a compact $0$-dimensional Hausdorff space, and assume further that $X$ has a finite number of isolated points.  If $C(X)$ is countably degree-$1$ saturated, then $CL(X)$ is countably saturated.  Moreover, if $X$ has no isolated points, then countable degree-$1$ saturation and countable saturation coincide for $C(X)$.
\end{teo}
\subsection*{Proof of Theorem \ref{theorem1}}
Countable saturation of $C(X)$ (for all formulas in the language of C*-algebras) implies saturation of the Boolean algebra, since being a projection is a weakly-stable relation, so every formula in $CL(X)$ can be rephrased in a formula in $C(X)$; to do so, write $\sup$ for $\forall$, $\inf$ for $\exists$, $\norm{x-y}$ for $x\neq y$, and so forth, restricting quantification to projections.  This establishes the first implication in Theorem \ref{theorem1}.  The second implication will require more effort.  To start, we will to need the following Proposition, relating elements of $C(X)$ to certain collections of clopen sets:
\begin{prop}\label{coding} Let $X$ be a compact $0$-dimensional space and $f\in C(X)_{\leq 1}$. Then there exists a countable collection of clopen sets $\tilde Y_f=\{Y_{n,f} : n \in \en\}$ which completely determines $f$, in the sense that for each $x \in X$, the value $f(x)$ is completely determined by $\{n\colon x\in Y_{n,f}\}$.
\end{prop}
\begin{proof}
Let $\ce_{m,1}=\{\frac{j_1+\sqrt{-1}j_2}{m}\colon j_1,j_2\in\zet\wedge \norm{j_1+\sqrt{-1}j_2}\leq m\}$.

For every $y\in\ce_{m,1}$ consider $X_{y,f}=f^{-1}(B_{1/m}(y))$. We have that each $X_{y,f}$ is a $\sigma$-compact open subset of $X$, so is a countable union of clopen sets $X_{y,f,1},\ldots,X_{y,f,n},\ldots\in CL(X)$. Note that $\bigcup_{y\in \ce_{m,1}}\bigcup_{n\in\en} X_{y,f,n}=X$. Let $\tilde X_{m,f}=\{X_{y,f,n}\}_{(y,n)\in\ce_{m,1}\times\en}\subseteq CL(X)$.

We claim that $\tilde X_f=\bigcup_m\tilde X_{m,f}$ describes $f$ completely. Fix $x\in X$. For every $m\in\en$ we can find a (not necessarily unique) pair $(y,n)\in\ce_{m,1}$ such that $x\in X_{y,f,n}$. Note that, for any $m,n_1,n_2\in\en$ and $y\neq z$, we have that $X_{y,f,n_1}\cap X_{z,f,n_2}\neq\emptyset$ implies $\abs{y-z}\leq\sqrt{2}/m$. In particular, for every $m\in\en$ and $x\in X$ we have \[2\leq \abs{\{y\in\ce_{m,1}\colon \exists n (x\in X_{y,f,n})\}}\leq 4.\] Let $A_{x,m}=\{y\in\ce_{m,1}\colon \exists n (x\in X_{y,f,n})\}$ and choose $a_{x,m} \in A_{x, m}$ to have minimal absolute value. Then $f(x)=\lim_m a_{x,m}$ so the collection $\tilde X_f$ completely describes $f$ in the desired sense.
\end{proof}

The above proposition will be the key technical ingredient in proving the second implication in Theorem \ref{theorem1}.  We will proceed by first obtaining the desired result under the Continuum Hypothesis, and then showing how to eliminate the set-theoretic assumption.

\begin{Lemma}\label{lem:CHqfsat}
Assume the Continuum Hypothesis.  Let $B$ be a countably saturated Boolean algebra of cardinality $2^{\aleph_0} = \aleph_1$.  Then $C(S(B))$ is countably saturated.
\end{Lemma}
\begin{proof}
Let $B' \preceq B$ be countable, and let $\mc{U}$ be a non-principal ultrafilter on $\en$.  By the uniqueness of countably saturated models of size $\aleph_1$, and the continuum hypothesis, we have $B'^{\mc{U}} \cong B$.  By Lemma \ref{lem:Bankston} we therefore have $C(S(B)) \cong C(S(B'))^{\mc{U}}$, and hence $C(S(B))$ is countably saturated.
\end{proof}

\begin{teo}\label{thm:CHqfsat}
Assume the Continuum Hypothesis.  Let $X$ be a compact Hausdorff $0$-dimensional space. If $CL(X)$ is countably saturated as a Boolean algebra, then $C(X)$ is quantifier free saturated.
\end{teo}
\begin{proof}
Let $\norm{P_n}=r_n$ be a collection of conditions, where each $P_n$ is a $2$-degree $*$-polynomial in $x_0,\ldots,x_n$, such that there is a collection $F=\{f_{n,i}\}_{n\leq i}\subseteq C(X)_{\leq 1}$, with the property that for all $i$ we have $\norm{P_n(f_{0,i},\ldots,f_{n,i})}\in (r_n)_{1/i}$ for all $n\leq i$.

For any $n$, we have that $P_n$ has finitely many coefficients. Consider $G$ the set of all coefficients of every $P_n$ and $L$ the set of all possible $2$-degree ${}^*$-polynomials in $F\cup G$. Note that for any $n\leq i$ we have that $P_{n}(f_{0,i},\ldots,f_{n,i})\in L$ and that $L$ is countable. For any element $f\in L$ consider a countable collection $\tilde X_f$ of clopen sets describing $f$, as in Proposition \ref{coding}.

Since $CL(X)$ is countably saturated and $2^{\aleph_0}=\aleph_1$ we can find a countably saturated Boolean algebra $B\subseteq CL(X)$ such that $\emptyset,X\in B$, for all $f\in L$ we have $\tilde X_f\subseteq B$, and $\abs{B}=\aleph_1$. 

Let $\iota \colon B\to CL(X)$ be the inclusion map.  Then $\iota$ is an injective Boolean algebra homomorphism and hence admits a dual continuous surjection $g_\iota\colon X\to S(B)$.

\begin{claim}
For every $f\in L$ we have that $\bigcup\tilde X_f=S(B)$.
\end{claim}

\begin{proof}
Recall that \[\bigcup \tilde X_f=X.\]
By compactness of $X$, there is a finite $C_{f}\subseteq \tilde X_{f}$ such that $\bigcup C_{f}=X$. In particular every ultrafilter on $B$ (i.e., a point of $S(B)$), corresponds via $g_\iota$ to an ultrafilter on $CL(X)$ (i.e., a point of $X$), and it has to contain an element of $C_{f}$. So $\bigcup \tilde X_{f}=S(B)$.
\end{proof}

From $g_\iota$ as above, we can define the injective map $\phi\colon C(S(B))\to C(X)$ defined as $\phi(f)(x)=f(g_\iota^{-1}(x))$.  Note that $\phi$ is norm preserving: Since $\phi$ is a unital $*$-homomorphism of C*-algebra we have that $\norm{\phi(f)}\leq\norm{f}$. For the converse, suppose that $x\in S(B)$ is such that $\abs{f(x)}=r$, and by surjectivity take $y\in X$ such that $g_\iota(y)=x$. Then \[\abs{\phi(f)(y)}=\abs{f(g_\iota(g_\iota^{-1}(x)))}=\abs{f(x)}.\]

For every $f\in L$ consider the function $f'$ defined by $\tilde X_f$ and construct the corresponding ${}^*$-polynomials $P_n'$.
\begin{claim}
\begin{enumerate}
\item\label{1} $f=\phi(f')$ for all $f\in L$.
\item \label{2}$\norm{P_n'(f_{0,i}',\ldots,f_{n,i}')}\in (r_n)_{1/i}$ for all $i$ and $n\leq i$.
\end{enumerate}
\end{claim}
\begin{proof}
Note that, since $f_{n,i}\in L$ and every coefficient of $P_n$ is in $L$, we have that $P_n(f_{0,i},\ldots,f_{n,i})\in L$. It follows that condition 1, combined with the fact that $\phi$ is norm preserving, implies condition 2.

Recall that $g=g_\iota$ is defined by Stone duality, and is a continuous surjective map $g\colon X\to Y$. In particular $g$ is a quotient map. Moreover by definition, since $X_{q,f,n}\in CL(Y)=B\subseteq CL(X)$, we have that if $x\in Y$ is such that $x\in X_{q,f,n}$ for some $(q,f,n)\in \qu\times L \times \en$, then for all $z$ such that $g(z)=x$ we have $z\in X_{q,f,n}$. Take $f$ and $x\in X$ such that $f(x)\neq\phi(f')(x)$. Consider $m$ such that $\abs{ f(x)-\phi(f')(x)}>2/m$. Pick $y\in\ce_{m,1}$ such that there is $k$ for which $x\in X_{y,f,k}$ and find $z\in Y$ such that $g(z)=x$. Then $z\in X_{y,f,k}$, that implies $f'(z)\in B_{1/m}(y)$ and so $\phi(f')(x)=f'(z)\in B_{1/m}(y)$ contradicting $\abs{f(x)-\phi(f')(x)}\geq 2/m$. 
\end{proof}
Consider now $\{\norm {P'_n(x_0,\ldots,x_n)}=r_n\}$. This type is consistent type in $C(S(B))$ by condition 2, and $C(S(B))$ is countably saturated by Lemma \ref{lem:CHqfsat}, so there is a total solution $\overline g$. Then $h_j=\phi(g_j)$ will be such that $\norm{P_n(\overline h)}=r_n$, since $\phi$ is norm preserving, proving quantifier-free saturation for $C(X)$.
\end{proof}

To remove the Continuum Hypothesis from Theorem \ref{thm:CHqfsat} we will show that the result is preserved by $\sigma$-closed forcing.  We first prove a more general absoluteness result about truth values of formulas.  For the necessary background in forcing, the reader can consult \cite[Chapter VII]{kunen:settheory}.  For more examples of absoluteness of model-theoretic notions, see \cite[Appendix]{Baldwin}.

Our result will be phrased in terms of truth values of formulas of \emph{infinitary} logic for metric structures.  Such a logic, in addition to the formula construction rules of the finitary logic we have been considering, also allows the construction of $\sup_n \phi_n$ and $\inf_n \phi_n$ as formulas when the $\phi_n$ are formulas with a total of finitely many free variables.  Two such infinitary logics have been considered in the literature.  The first, introduced by Ben Yaacov and Iovino in \cite{BenYaacovIovino}, allows the infinitary operations only when the functions defined by the formulas $\phi_n$ all have a common modulus of uniform continuity; this ensures that the resulting infinitary formula is again uniformly continuous.  The second, introduced by the first author in \cite{Eagle2014}, does not impose any continuity restriction on the formulas $\phi_n$ when forming countable infima or suprema; as a consequence, the infinitary formulas of this logic may define discontinuous functions.  The following result is valid in both of these logics; the only complication is that we must allow metric structures to be based on incomplete metric space, since a complete metric space may become incomplete after forcing.

\begin{Lemma}\label{absoluteness}
Let $M$ be a metric structure, $\phi(\overline{x})$ be a formula of infinitary logic for metric structures, and $\overline{a}$ be a tuple from $M$ of the appropriate length.  Let $\mathbb{P}$ be any notion of forcing.  Then the value $\phi^M(\overline{a})$ is the same whether computed in $V$ or in the forcing extension $V[G]$.
\end{Lemma}
\begin{proof}
The proof is by induction on the complexity of formulas; the key point is that we consider the structure $M$ in $V[G]$ as the same set as it is in $V$.  The base case of the induction is the atomic formulas, which are of the form $P(\overline{x})$ for some distinguished predicate $P$.  In this case since the structure $M$ is the same in $V$ and in $V[G]$, the value of $P^M(\overline{a})$ is independent of whether it is computed in $V$ or $V[G]$.

The next case is to handle the case where $\phi$ is $f(\psi_1, \ldots, \psi_n)$, where each $\psi_i$ is a formula and $f : [0, 1]^n \to [0, 1]$ is continuous.  Since the formula $\phi$ is in $V$, so is the function $f$.  By induction hypothesis each $\psi_i^M(\overline{a})$ can be computed either in $V$ or $V[G]$, and so the same is true of $\phi^M(\overline{a}) = f(\psi_1^M(\overline{a}), \ldots, \psi_n^M(\overline{a}))$.  A similar argument applies to the case when $\phi$ is $\sup_n\psi_n$ or $\inf_n\psi_n$.

Finally, we consider the case where $\phi(\overline{x}) = \inf_y \psi(\overline{x}, y)$ (the case with $\sup$ instead of $\inf$ is similar).  Here we have that for every $b \in M$, $\psi^M(\overline{a}, b)$ is independent of whether computed in $V$ or $V[G]$ by induction.  In both $V$ and $V[G]$ the infimum ranges over the same set $M$, and hence $\phi^M(\overline{a})$ is also the same whether computed in $V$ or $V[G]$.
\end{proof}

We now use this absoluteness result to prove absoluteness of countable saturation under $\sigma$-closed forcing.

\begin{prop}\label{saturationAbsolute}
Let $\mathbb{P}$ be a $\sigma$-closed notion of forcing.  Let $M$ be a metric structure, and let $\Phi$ be a set of (finitary) formulas.  Then $M$ is countably $\Phi$-saturated in $V$ if and only if $M$ is countably $\Phi$-saturated in the forcing extension $V[G]$.
\end{prop}
\begin{proof}
First, observe that since $\mathbb{P}$ is $\sigma$-closed, forcing with $\mathbb{P}$ does not introduce any new countable set.  In particular, the set of types which must be realized for $M$ to be countably $\Phi$-saturated are the same in $V$ and in $V[G]$.

Let $\mathbf{t}(\overline{x})$ be a set of instances of formulas from $\Phi$ with parameters from a countable set $A \subseteq M$.  Add new constants to the language for each $a \in A$, so that we may view $\mathbf{t}$ as a type without parameters.  Define
\[\phi(\overline{x}) = \inf\{\psi(\overline{x}) : \psi \in \mathbf{t}\}.\]
Note that $\phi^M(\overline{a}) = 0$ if and only if $\overline{a}$ satisfies $\mathbf{t}$ in $M$.  This $\phi$ is a formula in the infinitary logic of \cite{Eagle2014}.  By Lemma \ref{absoluteness} for any $\overline{a}$ from $M$ we have that $\phi^M(\overline{a}) = 0$ in $V$ if and only if $\phi^M(\overline{a}) = 0$ in $V[G]$.  As the same finite tuples $\overline{a}$ from $M$ exist in $V$ and in $V[G]$, this completes the proof.
\end{proof}

Finally, we return to the proof of Theorem \ref{theorem1}.  All that remains is to show:

\begin{Lemma}\label{removingCH}
The Continuum Hypothesis can be removed from the hypothesis of Theorem \ref{thm:CHqfsat}. 
\end{Lemma}

\begin{proof}
Let $X$ be a $0$-dimensional compact space such that $CL(X)$ is countably saturated, and suppose that the Continuum Hypothesis fails.  Let $\mathbb{P}$ be a $\sigma$-closed forcing which collapses $2^{\aleph_0}$ to $\aleph_1$ (see \cite[Chapter 7, \S 6]{kunen:settheory}).  Let $A = C(X)$.  Observe that since $\mathbb{P}$ is $\sigma$-closed we have that $X$ remains a compact $0$-dimensional space in $V[G]$, and we still have $A = C(X)$ in $V[G]$.  By Proposition \ref{saturationAbsolute} we have that $CL(X)$ remains countably saturated in $V[G]$.  Since $V[G]$ satisfies the Continuum Hypothesis we can apply Theorem \ref{thm:CHqfsat} to conclude that $A$ is countably quantifier-free saturated in $V[G]$, and hence also in $V$ by Proposition \ref{saturationAbsolute}.
\end{proof}

With the continuum hypothesis removed from Theorem \ref{thm:CHqfsat}, we have completed the proof of Theorem \ref{theorem1}.  It would be desirable to improve this result to say that if $CL(X)$ is countably saturated then $C(X)$ is countably saturated.  We note that if the map $\phi$ in Theorem \ref{thm:CHqfsat} could be taken to be an elementary map then the same proof would give the improved conclusion.
\subsection*{Proof of Theorem \ref{theorem2}}
We now turn to the proof of Theorem \ref{theorem2}. We start from the easy direction:

\begin{prop}\label{totaldisc}
If $X$ is a $0$-dimensional compact space with finitely many isolated points such that $C(X)$ is countably degree-$1$ saturated, then the Boolean algebra $CL(X)$ is countably saturated.
\end{prop}
\begin{proof}
Assume first that $X$ has no isolated points.  In this case we get that $CL(X)$ is atomless, so it is enough to see that $CL(X)$ satisfies the equivalent condition of Theorem \ref{teo:satbooleanalgebra}.

Let $Y < Z$ be directed such that $\abs{Y} + \abs{Z} < \aleph_1$.  Assume for the moment that both $Y$ and $Z$ are infinite.  Passing to a cofinal increasing sequence in $Z$ and a cofinal decreasing sequence in $Y$, we can suppose that $Z=\{U_n\}_{n\in\en}$ and $Y=\{V_n\}_{n\in\en}$, where
\[U_1\subsetneq\ldots\subsetneq U_n\subsetneq U_{n+1}\subsetneq\ldots\subsetneq V_{n+1}\subsetneq V_n\subsetneq\ldots\subsetneq V_1.\]
If $\bigcup_{n\in\en} U_n=\bigcap_{n\in\en} V_n$ then $\bigcup_{n\in\en} U_n$ is a clopen set, so by the remark following the proof of Lemma \ref{abcondition1}, we have a contradiction to the countable degree-$1$ saturation of $C(X)$. 

For each $n \in \en$, let $p_n=\chi_{U_n}$ and $q_n=\chi_{V_n}$, where $\chi_A$ denotes the characteristic function of the set $A$. Then 
\[p_1<\ldots < p_n<p_{n+1}<\ldots<q_{n+1}<q_n<\ldots<q_1\] 
and by countable degree-$1$ saturation there is a positive $r$ such that $p_n<r<q_n$ for every $n$. In particular $A=\{x\in X\colon r(x)=0\}$ and $C=\{x\in X\colon r(x)=1\}$ are two disjoint closed sets such that $\overline{\bigcup_{n\in\en} U_n}\subseteq C$ and $\overline{X\setminus\bigcap_{n\in\en} V_n}\subseteq A$. We want to find a clopen set $D$ such that $A\subseteq D\subseteq X\setminus C$. For each $x\in A$ pick $W_x$ a clopen neighborhood contained in $X\setminus C$. Then $A\subseteq\bigcup_{x\in A}W_x$. By compactness we can cover $A$ with finitely many of these sets, say $A\subseteq \bigcup_{i\leq n}W_{x_i}\subseteq X\setminus C$, so $D = \bigcup_{i\leq n}W_{x_i}$ is the desired clopen set.

Essentially the same argument works when either $Y$ or $Z$ is finite.  We need only change some of the inequalities from $<$ with $\leq$, noting that a finite directed set has always a maximum and a minimum.

If $X$ has a finite number of isolated points, write $X = Y \cup Z$, where $Y$ has no isolated points and $Z$ is finite.  Then $C(X) = C(Y) \oplus C(Z)$ and $CL(X) = CL(Y) \oplus CL(Z)$.  The above proof shows that $CL(Y)$ is countably saturated, and $CL(Z)$ is saturated because it is finite, so $CL(X)$ is again saturated.
\end{proof}

To finish the proof of Theorem \ref{theorem2} it is enough to show that when $X$ has no isolated points the theory of $X$ admits elimination of quantifiers.  By Corollary \ref{cor:elementaryEquivalence} we have that $C(X) \equiv C(\beta\en\setminus\en)$ for such $X$, so it suffices to show that the theory of $C(\beta\en\setminus\en)$ eliminates quantifiers.  We thank Ilijas Farah and Bradd Hart for giving permission to include their unpublished proof of this result.

\begin{remark}
We point out that is not surprising that the theory of abelian real rank zero C*-algebra without minimal projections has quantifier elimination. In fact it is known that in this case the theory of the associated Boolean algebras has quantifier elimination. Conversely, if the Boolean algebra $B$ is not atomless, the theory of $B$ does not have quantifier elimination (in the language of Boolean algebras), and consequently when $X$ is compact $0$-dimensional and with densely many isolated points, the theory of $C(X)$ does not have elimination of quantifiers.  To see this, note that quantifier elimination implies model completeness, that is, if $A \models T$, $B \models T$ and $B \subseteq A$, then $B \preceq A$. In particular let $A=C(X)$ with $X$ as above, and let $T$ be the theory of real rank zero unital C*-algebra such that any projection has a minimal projection below it. Identify two different isolated points with each other with a function $f$ and consider the quotient space $Y$. Then $C(Y)\models T$, and since $f\colon X\to Y$ is surjective we have an embedding of $C(Y)$ into $C(X)$. Since there is a minimal projection in $C(Y)$ that is not minimal in $C(X)$, this embedding is not elementary.
\end{remark}

\begin{defin}
Let $a_1,\ldots,a_n\in C(X)$ (more generally, one can consider commuting operators on some Hilbert space $H$).  We say that $\overline a=(a_1,\ldots,a_n)$ is \defined{non-singular} if the polynomial $\sum_{i=1}^n a_ix_i=I$ has a solution $x_1,\ldots, x_n$ in $C(X)$.

We define the \defined{joint spectrum} of $a_1, \ldots, a_n$ to be 
\[j\sigma(\overline a)=\{\overline\lambda\in\ce^n\colon (\lambda_1-a_1,\ldots,\lambda_n-a_n)\text{ is singular}\}\]
\end{defin}

\begin{prop}
Fix $a_1,\ldots,a_n\in C(X)$. Then $\overline \lambda\in j\sigma(\overline a)$ if and only if $\sum_{i\leq n}\abs{\lambda_i-a_i}$ is not invertible.
\end{prop}
\begin{proof}
We have that $\overline\lambda\in j\sigma(\overline a)$ if and only if there is $x\in X$ such that $a_i(x)=\lambda_i$ for all $i\leq n$.
In particular, $\overline\lambda\in j\sigma(\overline a)$ if and only if $0\in\sigma(\sum \abs{\lambda_i-a_i})$ if and only if there is $x$ such that $\sum_{i\leq n}\abs{\lambda_i-a_i}(x)=0$. Since each $\abs{\lambda_i-a_i}$ is positive we have that this is possible if and only if there is $x$ such that for all $i\leq n$, $\abs{\lambda_i-a_i}(x)=0$.
\end{proof}
\begin{prop}The joint spectrum of an abelian C*-algebra $A$ is quantifier free-definable.
\end{prop}
\begin{proof}
First of all recall that, when $\overline a=a$, then $j\sigma(\overline a)=\sigma(a)$, hence the two definitions coincide for elements. We want to define a quantifier-free definable function $F\colon A\times\ce\to[0,1]$ such that $F(a,\lambda)=0$ if and only if $\lambda\in\sigma(a)$. Since we showed that $\overline\lambda\in\sigma(\overline a)$ if and only if $0\in\sigma(\sum_{i\leq n}\abs{\lambda_i-a_i})$, so, in light of this, we can define a function \[F_n\colon A^n\times\ce^n\to [0,1]\] as $F_n(\overline a,\overline\lambda)=F(\sum \abs{\lambda_i-a_i},0)$, hence we have that $F_n(\overline a,\overline \lambda)=0$ if and only if $\overline\lambda\in j\sigma(\overline a)$, that implies that the joint spectrum of $\overline a\in A^n$ is quantifier-free definable.

To define $\sigma(a)$, recall that, for $f\in A$, the absolute value of $f$ is quantifier-free definable as $\abs{f}=\sqrt{ff^*}$, and for a self-adjoint $f\in A$, its positive part is quantifier-free definable as the function $f_+=\max(0,f)$.  Then $F(a,\lambda)=\abs{1-\norm{(1-\abs{a-\lambda\cdot 1})_+}}$ is the function we were seeking.
\end{proof}
\begin{teo}\label{thm:quantifierElimination}
The theory of $C(\beta\en\setminus\en)$ has quantifier elimination. Consequently the theory of real rank zero abelian C*-algebras without minimal projections has quantifier elimination.
\end{teo}
\begin{proof}
It is enough to prove that for any $n\in\en$ and $\overline a, \overline b\in C(\beta\en\setminus\en)^n$ that have the same quantifier-free type over $\emptyset$ there is an automorphism of $C(\beta\en\setminus\en)$ sending $a_i$ to $b_i$, for all $i\leq n$.

Since $\overline a$ and $\overline b$ have the same quantifier-free type, we have that $K=j\sigma(\overline a)=j\sigma(\overline b)$. Consider $D$ be a countable dense subset of $K$ and pick $f_1,\ldots,f_n,g_1,\ldots,g_n\in C(\beta\en)=\ell^\infty(\en)$ such that $\forall (d_1,\ldots,d_n)\in D$ we have that $F_d=\{m\in\en\colon \forall i\leq n (f_i(m)=d_i)\}$ and $G_d=\{m\in\en\colon\forall i\leq n(g_i(m)=d_i) \}$ are infinite, $\pi(f_i)=a_i$, $\pi(g_i)=b_i$ and, for $m\in\en$ we have that $(f_1(m),\ldots,f_n(m)),(g_1(m),\ldots,g_n(m))\in D$.

In particular we have that $\en=\bigcup_{d\in D} F_d=\bigcup_{d\in D}G_d$ and that for all $d\neq d'$ we have $F_d\cap F_{d'}=\emptyset=G_d\cap G_{d'}$, then there is a permutation $\sigma$ on $\en$ (that induces an automorphism of $C(\beta\en\setminus\en)$) such that $f_i\circ\sigma=g_i$ for all $i\leq n$.
\end{proof}

The proof of Theorem \ref{theorem2} is now complete by combining Theorem \ref{theorem1}, Proposition \ref{totaldisc}, and Theorem \ref{thm:quantifierElimination}.

\section{Conclusions and questions}
In light of our result related to the Breuer ideal of a II$_1$ factor, we can ask whether or not this quotient structure carries more saturation than countable degree-$1$ saturation, and if the degree of saturation may depend on the structure of the II$_1$ factor itself. The proof of the failure of quantifier-free saturation for the Calkin algebra, as stated in \cite{farah2011countable}, involved the notion of the abelian group $\operatorname{Ext}$, a notion which has not been developed for quotients with an ideal that is not $\sigma$-unital, so it cannot be easily modified to obtain a similar result for our case (and in general, for the case of a tracial Von Neumann algebra modulo the ideal of finite projections).

It is natural, then, to state the following list of questions:
\begin{Question}Let $M$ be a II$_1$-factor, and $H$ separable. Is $M\overline\bigotimes\mathcal B(H)/\mathcal K$ quantifier-free saturated? Is fully saturated? Does this degree of saturation depend on the structure of $M$?
\end{Question}

Heading in the direction taken in the second part of section \ref{section:Calkin}, it is natural to consider saturation for the algebras $\mathcal C_\alpha$, when $\alpha \geq 1$ is an uncountable cardinal. From a set-theoretical point of view, saturation of the Boolean algebra $\mathcal P(\aleph_\alpha)/Fin$ was proven using quantifier elimination of the theory of atomless Boolean algebras, an advantage that we do not have here. Recalling that  each ideal of $\mathcal B_\alpha$ is generated by operators whose density is less than $\aleph_\beta$, for some $\beta\leq\alpha$, we note that if $\alpha\geq 1$ no ideal of $\mathcal B_\alpha$ satisfies the hypothesis of Theorem \ref{thm:FactorCtbleSat}, hence it is conceivable to ask whether or not the degree of saturation reflects to different quotients of $\mathcal B_\alpha$. This question in justified by the fact that, whenever $\alpha\leq\beta$ are ordinals, it is known that the Boolean algebra $\mathcal P(\aleph_\beta)/\{a\in\mathcal P(\aleph_\beta)\mid |a|<\aleph_\alpha\}$ is countably saturated (as a a consequence of \cite{Mija79}).
\begin{Question}Let $\alpha\geq 1$ be an ordinal. Which quotients of $\mathcal B_\alpha$ are countably $1$-degree saturated?
\end{Question}
Concerning the theories of the generalized Calkin algebras, we have the following question, which is resolved for $\alpha, \beta < \omega^\omega\cdot 2$ by Theorem \ref{teo:generalizedcalkin}.
\begin{problem}
Characterize pairs of ordinals $(\alpha,\beta)$ such that $\mathcal C_\alpha\equiv\mathcal C_\beta$ (or $\mathcal B_\alpha\equiv \mathcal B_\beta)$.
\end{problem}
In light of Theorem \ref{teo:generalizedcalkin}, it seems reasonable to conjecture that $\mathcal C_\alpha \equiv \mathcal C_\beta$ if and only if $\alpha \equiv \beta$. Note that $\mathcal B_\alpha\equiv\mathcal B_\beta$ implies $\mathcal C_\alpha\equiv \mathcal C_\beta$, since finite projections (and therefore the compact operators) form a definable set and so the Calkin algebra of weight $\alpha$ is definable inside $\mathcal B_\alpha$.

Two of our results, namely Theorems \ref{teo:generalizedcalkin} and \ref{boolean0dim}, express elementary equivalence of certain classes of C*-algebras in terms of elementary equivalence of associated discrete structures.  The following question, which was suggested by Ilijas Farah, is of the same flavour, and is moreover motivated by Elliott's classification  of AF algebras using $K$-theory in \cite{Elliott}.  We note that, unlike the examples considered here, the class of AF algebras is described by an omitting types property rather than by a theory (see \cite[Theorem 2]{Carlson}).  In particular, the method of proof of Theorem \ref{boolean0dim} is unavailable here.

\begin{Question}
If $A$ and $B$ are AF algebras such that $K_0(A) \equiv K_0(B)$ as (discrete) ordered abelian groups, is it true that $A \equiv B$?
\end{Question}

Turning now to the abelian setting, we can see that the $0$-dimensional case is largely settled. It remains open whether the implications of Theorem \ref{theorem1} can be reversed.  It would also be desirable to extend the final part of Theorem \ref{theorem2} to the case where the space has isolated points.  In the case where $X$ is $0$-dimensional and without isolated points, we used quantifier elimination of $C(X)$ to show that saturation of $CL(X)$ implies the saturation of $C(X)$.  This method does not generalize, as recent work of a Fields Institute undergraduate research group has recently shown that if $X$ has isolated points then $C(X)$ does not have quantifier elimination \cite{Fields}.  The same group has also shown that a large class of other spaces also have algebras without quantifier elimination, including all finite-dimensional manifolds.

There are many open questions regarding $C(X)$ when $X$ is not $0$-dimensional, in particular when $X$ has finite positive dimension.  In this case the Boolean algebra of clopen sets cannot be used as an invariant for the theory of $C(X)$, but it is natural to ask whether or not there is another well-known discrete invariant that characterizes the theory of the abelian C*-algebra associated to $X$. Regarding saturation, the situation seems to be much more complicated.  Recently Farah and Shelah \cite{farah2014rigidity} showed the existence of an $X\subseteq\er^2$ whose corona is countably degree-$1$ saturated but not quantifier free saturated.  The following question is still open:
\begin{Question}\label{qfnotfull} Does there exist a C*-algebra $A$ such that $A$ is quantifier-free saturated but $A$ is not countably saturated?
\end{Question}
As a particular case, we would like to mention a question that has been open for some time:
\begin{Question}\label{betaern}
Is $C(\beta\er^n\setminus\er^n)$ countably saturated, for $n\geq 2$?
\end{Question}
It is known that this algebra is countably degree-$1$ saturated. A positive answer to this question would solve a long-standing set theoretical problem, namely, would prove that under the Continuum Hypothesis there are $2^{\omega_1}$ many autohomeomorphisms of $\beta\er^n\setminus\er^n$.
\bibliographystyle{amsplain}
\bibliography{Saturation}

\end{document}